\documentclass{amsart}
\usepackage{amssymb}
\usepackage[T1]{fontenc}
\usepackage[utf8]{inputenc}
\usepackage{amssymb}
\usepackage{graphicx}

\usepackage{ifpdf}
\ifpdf
\else
\fi

\newtheorem{theorem}{Theorem}[section]
\newtheorem{lemma}[theorem]{Lemma}
\newtheorem{proposition}[theorem]{Proposition}
\theoremstyle{remark}
\numberwithin{equation}{section}
\newtheorem{remark}[theorem]{Remark}

\newtheorem{example}[theorem]{Example}
\begin{document}

\title[Leading terms of relations for standard modules of  $C_{n}\sp{(1)}$]{Leading terms of relations
for standard modules of affine Lie algebras $C_{n}\sp{(1)}$}
\author{Mirko Primc and Tomislav \v Siki\' c}
\address{Mirko Primc, University of Zagreb, Faculty of Science,  Croatia}
\email{primc@math.hr}
\address{Tomislav \v{S}iki\'{c}, University of Zagreb, Faculty of Electrical Engineering and  Com- \mbox{\hskip 7.8em} puting, Croatia}
\email{tomislav.sikic@fer.hr}

\subjclass[2000]{Primary 17B67; Secondary 17B69, 05A19.\\
\indent Partially supported by Croatian Science Foundation, Project 2634}

\begin{abstract}
In this paper we give a combinatorial parametrization of leading
terms of defining relations for level $k$ standard modules for affine Lie
algebra of type $C_{n}\sp{(1)}$. Using this parametrization we conjecture colored Rogers-Ramanujan type combinatorial identities
for $n\geq 2$ and $k\geq 2$; the identity in the case $n=k=1$ is equivalent to one of Capparelli's identities.
\end{abstract}
\maketitle

\section{Introduction}

Famous Rogers-Ramanujan identities are two analytic identities, for $a=0$ or $1$,
\begin{align}\label{E: analytic RR}
\prod_{m \geq 0}\frac{1}{(1-q^{5m+1+a})(1-q^{5m+4-a})}&= \sum_{m\geq0}\frac{q^{m^2+am}}{(1-q)(1-q^{2})\cdots (1-q^m)}.
\end{align}
If, for $a=0$, we expand both sides in Taylor series, then the  coefficient of $q^m$ obtained from the product side can be interpreted
as a number of partitions of $m$ with parts congruent $\pm 1\textrm{ mod }5$. On the other side, the coefficient of $q^m$ obtained from the sum side
can be interpreted as a number of partitions of $m$ such that a difference between two consecutive parts is at least two.
We can write a partition as
\begin{equation*}
\sum_{j\geq 1}jf_j\qquad\text{or}\qquad 1\sp{f_1}2\sp{f_2}3\sp{f_3}\dots,
\end{equation*}
meaning that the part $j$ appears $f_j$ times in the partition, and $f_j= 0$ for all but finitely many $j$. With this notation the
difference two condition between two consecutive parts can be written as
\begin{equation}\label{E: difference condition 1}
f_{j}+f_{j+1}\leq 1,\qquad j\geq 1,
\end{equation}
and two ways of expressing the coefficient of $q^m$ in Taylor series of (\ref{E: analytic RR}) can be stated as combinatorial Rogers-Ramanujan identity
\begin{equation*}
\#\{m=\sum jf_j\mid j\equiv \pm 1(\textrm{mod }5)\}=\#\{m=\sum jf_j\mid f_{j}+f_{j+1}\leq 1\}.
\end{equation*}
Analytic Rogers-Ramanujan identities have Gordon-Andrews-Bressoud's generalization
(cf. \cite{G}, \cite{A1}, \cite{A2}, \cite{B1}, \cite{B2}).
These identities also have a combinatorial interpretation, but in general it is not so easy to interpret sum sides of
analytic identities as generating functions for a number of partitions satisfying certain difference conditions among parts.

In 1980's Rogers-Ramanujan-type identities appeared in statistical physics and in representation theory of affine Kac-Moody Lie algebras.
This led to two lines of intensive research and numerous generalizations of both analytic and combinatorial identities,
the reader may consult, for example, the papers \cite{AKS}, \cite{BM}, \cite{BMS}, \cite{CLM}, \cite{FJLMM}, \cite{FS}, \cite{F},
\cite{FQ}, \cite{Ge}, \cite{JMS}, \cite{W} and the references therein.

J. Lepowsky and S. Milne discovered in \cite{LM} that the product sides of Gordon-Andrews-Bressoud identities, multiplied with
certain fudge factor $F$, are principally specialized characters of standard modules for affine
Lie algebra $\widehat{\mathfrak{sl}}_2$.  Lepowsky and R. L. Wilson
realized that the factor $F$ is a character of the Fock space for the principal Heisenberg subalgebra of  $\widehat{\mathfrak{sl}}_2$,
and that the sum sides of Gordon-Andrews-Bressoud identities are the principally specialized characters of the vacuum spaces of standard modules
for the action of principal Heisenberg subalgebra.
In a series of papers---see \cite{LW} and the references therein---Lepowsky and Wilson discovered vertex operators in the principal picture
on standard $\widehat{\mathfrak{sl}}_2$-modules and constructed bases of vacuum spaces for the principal Heisenberg subalgebra
parametrized by partitions satisfying certain difference 2 conditions. Very roughly speaking, in the Rogers-Ramanujan case, the vacuum
space $\Omega$ is spanned by monomial vectors of the form
\begin{equation}\label{E: LW argument 1}
Z(-s)\sp{f_s}\dots Z(-2)\sp{f_2}Z(-1)\sp{f_1}v_0,\quad s\geq 0, \ f_j\geq 0,
\end{equation}
where $v_0\in\Omega$ is a highest weight vector and $Z(j)$ are certain $\mathcal Z$-operators. The degree of monomial vector
(\ref{E: LW argument 1}) is
\begin{equation*}
-m=-\sum_{j=1}\sp s jf_j.
\end{equation*}
Since $\mathcal Z$-operators $Z(j)$ on $\Omega$ satisfy certain relations, roughly of the form
\begin{equation}\label{E: LW relations}
\begin{aligned}
&Z(-j)Z(-j)+2\sum_{i>0}Z(-j-i)Z(-j+i)\approx 0,\\
&Z(-j-1)Z(-j)+\sum_{i>0}Z(j-1-i)Z(-j+i)\approx 0
\end{aligned}
\end{equation}
(see  \cite{LW} and  \cite{MP1} for precise formulation), we may replace the {\it leading terms}
\begin{equation}\label{E: LW argument 3}
Z(-j)Z(-j) \ \text{and} \  Z(-j-1)Z(-j)
\end{equation}
of relations (\ref{E: LW relations}) with ``higher terms'' $Z(-j-i)Z(-j+i)$ and $Z(-j-1-i)Z(-j+i)$, $i>0$, and reduce the spanning set
(\ref{E: LW argument 1}) of $\Omega$ to a spanning set
\begin{equation}\label{E: LW argument 4}
Z(-s)\sp{f_s}\dots Z(-2)\sp{f_2}Z(-1)\sp{f_1}v_0,\quad s\geq 0,\  f_{j}+f_{j+1}\leq 1\ \text{for all} \  j\geq 1.
\end{equation}
By invoking the product formula for principally specialized character of $\Omega$ and Rogers-Ramanujan identities, we see that vectors
in the spanning set (\ref{E: LW argument 4}) are in fact a basis of $\Omega$. In such a way Gordon-Andrews-Bressoud identities also
appear for low level representations of different affine Lie algebras or for representations of low rank affine Lie algebras
(see  \cite{BoM}, \cite{M}, \cite{M1}---\cite{M4}, \cite{X}), and S. Capparelli has found new combinatorial identities (see \cite{C} and \cite{A3}).
The analogous construction in the homogeneous picture is obtained in \cite{LP}.

The results of this paper are closely related to a similar construction of combinatorial bases for standard $\widehat{\mathfrak{sl}}_2$-modules
obtained in \cite{MP2} and \cite{MP3} by A. Meurman and the first author,  and independently in \cite{FKLMM}
by   B. Feigin, R. Kedem, S. Loktev, T. Miwa and E. Mukhin. The starting point is a basis of  $\widehat{\mathfrak{sl}}_2$
\begin{equation}\label{E: MP argument 1}
c\quad\text{and}\quad x(j), h(j), y(j),\quad j\in\mathbb Z,
\end{equation}
where $\{x, h, y\}$ is the standard basis of ${\mathfrak{sl}}_2$, and the corresponding Poincar\' e-Birkhoff-Witt monomial spanning set of level $k$ standard $\widehat{\mathfrak{sl}}_2$-module $L(k\Lambda_0)$
\begin{equation}\label{E: MP argument 2}
y(-s)\sp{c_s}\dots y(-2)\sp{c_2}h(-2)\sp{b_2}x(-2)\sp{a_2}y(-1)\sp{c_1}h(-1)\sp{b_1}x(-1)\sp{a_1}v_0,\quad s\geq 0,
\end{equation}
with $a_j, b_j, c_j\geq 0$. The spanning set (\ref{E: MP argument 2})  is analogous to the spanning set  (\ref{E: LW argument 1}).
We have a relation
\begin{equation}\label{E: MP argument 3}
\sum_{j_1+\cdots +j_{k+1}=m}x(j_1)\cdots
x(j_{k+1})=0
\end{equation}
with the leading term
\begin{equation}\label{E: MP argument 4}
 x(-j-1)^b  x(-j)^a
\end{equation}
with $a+b=k+1$ and $(-j-1)b+(-j)a=m$. This is analogous to (\ref{E: LW relations}) and (\ref{E: LW argument 3}), and we can reduce the spanning set
(\ref{E: MP argument 2}) to a smaller spanning set satisfying the difference condition
$$
a_{j+1}+a_{j}\leq k,
$$
but this spanning set is not a basis of  $L(k\Lambda_0)$.
All the relations needed to reduce (\ref{E: MP argument 2}) to a basis of  $L(k\Lambda_0)$ are obtained from (\ref{E: MP argument 3}) by
the adjoint action of ${\mathfrak{sl}}_2$, and all the leading terms are obtained by the adjoint action of ${\mathfrak{sl}}_2$ on  (\ref{E: MP argument 4})
\begin{equation}\label{E: MP argument 5}
\begin{aligned}
&x(-j-1)^b h(-j)^{a_2} x(-j)^{a_1},\quad a_1+a_2 = a , \\
&x(-j-1)^b y(-j)^{a_2} h(-j)^{a_1},\quad  a_1+a_2 = a ,\\
&h(-j-1)^{b_1} x(-j-1)^{b_2} y(-j)^{a},\quad  b_1+b_2 = b , \\
&y(-j-1)^{b_1} h(-j-1)^{b_2} y(-j)^{a},\quad  b_1+b_2 = b
\end{aligned}
\end{equation}
(see  (\ref{E:12.16}) bellow). By using these relations and their leading terms we can reduce a spanning set  (\ref{E: MP argument 2})
to a smaller spanning set of  $L(k\Lambda_0)$ satisfying the difference conditions
\begin{equation}\label{E: MP argument 6}
\begin{aligned}
&a_{j+1}+b_{j}+a_{j}\leq k, \\
&a_{j+1}+c_{j}+b_{j}\leq k, \\
&b_{j+1}+a_{j+1}+c_{j}\leq k, \\
&c_{j+1}+b_{j+1}+c_{j}\leq k. \\
\end{aligned}
\end{equation}
In \cite{FKLMM} and \cite{MP3} it is proved, by different methods, that this spanning set is a basis of  $L(k\Lambda_0)$. This is analogous
to difference conditions for the basis  (\ref{E: LW argument 4}). The degree of monomial vector (\ref{E: MP argument 2})
satisfying the difference conditions  (\ref{E: MP argument 6}) is
\begin{equation*}
-m=-\sum_{j\geq 1}ja_j-\sum_{j\geq 1}jb_j-\sum_{j\geq 1}jc_j,
\end{equation*}
so we are naturally led to interpret monomial basis vectors (\ref{E: MP argument 2}) in terms of colored partitions with parts $j$ in three colors:
$x$, $h$ and $y$ (cf. \cite{A}, \cite{AA}, \cite{JMS}).

In this paper we give a combinatorial description of (some) leading terms of relations for level $k\geq 1$ standard
modules $L(k\Lambda_0)$ of affine Lie algebras $C_{n}\sp{(1)}$ for all $n\geq 2$. For the set of indices
$\{1,2,\cdots ,n,\underline{n},\cdots ,\underline{2},\underline{1}\}$ we parametrize a basis of the Lie algebra of type $C_n$ as
\begin{equation*}
B=\{X_{ab}\mid b\in\{1,2,\cdots ,n,\underline{n},\cdots ,\underline{2},\underline{1}\},\ a\in\{1,\cdots ,b\}\} .
\end{equation*}
We visualize $B$ as a triangle---for $n=3$ we have
$$\begin{array}{cccccc}
11 &  &&  & & \\
12 & 22 & & & & \\
13 & 23 & 33 & & & \\
1\underline{3} & 2\underline{3} & 3\underline{3} & \underline{3}\underline{3}  & & \\
1\underline{2} & 2\underline{2} & 3\underline{2} & \underline{3}\underline{2}  & \underline{2}\underline{2}& \\
1\underline{1} & 2\underline{1} & 3\underline{1} & \underline{3}\underline{1}  & \underline{2}\underline{1} & \underline{1}\underline{1}.
\end{array}
$$
For any point $rr$ on the diagonal we observe two triangles with this point in common. For example, for $r=\underline{3}$ we have triangles
$$\begin{array}{cccccc}
11 &  &&  & & \\
12 & 22 & & & & \\
13 & 23 & 33 & & & \\
1\underline{3} & 2\underline{3} & 3\underline{3} & \underline{3}\underline{3}  & & \\
&& & \underline{3}\underline{2}  & \underline{2}\underline{2}& \\
&&& \underline{3}\underline{1}  & \underline{2}\underline{1} & \underline{1}\underline{1}.
\end{array}
$$
In each of these triangles we observe a cascade---a ``staircase'' going downwards from the right to the left with a given multiplicity at
each point. For example, we have two cascades, $\mathcal B$ and $\mathcal A$,
$$\begin{array}{cccccc}
\cdot &  &&  & & \\
\cdot & \cdot & & & & \\
\cdot & 3 & 2 & & & \\
5 & 0& \cdot & \cdot & & \\
&& & \cdot  &1& \\
&&&1 &2& \cdot
\end{array}
$$
in triangles with the common point $ \underline{3}\underline{3}$. For two such cascades and some $j\in\mathbb Z$  we can write
the leading term of a relation---in our example it is the monomial
$$
X_{33}(-j-1)\sp2X_{23}(-j-1)\sp3 X_{2\underline{3} }(-j-1)\sp0 X_{1\underline{3}}(-j-1)\sp5 \,X_{ \underline{2}\underline{2}}(-j)\sp1
X_{ \underline{2}\underline{1} }(-j)\sp2 X_{ \underline{3}\underline{1} }(-j)\sp1.
$$
This monomial is the leading term of some relation for
$$
k+1=(2+3+0+5)+(1+2+1),
$$
i.e., for standard module $L(13\Lambda_0)$. By using these relations and their leading terms we can reduce a PBW-spanning set
$$
\prod_{ab\in B, \,j>0}X_{ab}(-j)\sp{m_{ab;j}}\,v_0
$$
of $L(k\Lambda_0)$ to a smaller spanning set satisfying difference conditions
\begin{equation}\label{E: Cn argument 1}
\sum_{ab\in\mathcal B} m_{ab;j+1}+\sum_{ab\in\mathcal A} m_{ab;j}\leq k
\end{equation}
for any two cascades as above. 
{\it We conjecture that this spanning set is in fact a basis.}

  If we interpret ${\mathfrak{sl}}_2$ as
type $C_1$ Lie algebra, then our list of leading terms coincides with (\ref{E: MP argument 5}), and our difference conditions
(\ref{E: Cn argument 1}) coincide with difference conditions (\ref{E: MP argument 6}),
so by  \cite{FKLMM} and \cite{MP3} the spanning set is a basis. For level $k=1$ and $C_2\sp{(1)}$ the spanning set is a basis by \cite{S},
and we show in \cite{P\v S} that  the spanning set is a basis for $k=1$ and all $n\geq 2$. These
are the only cases in which our conjecture is proved.
Numerical evidence supports our conjecture in the case $n=k=2$ (see Example 7.1 in the last
section).

It took us quite a while to understand the combinatorial parametrization of leading terms even for the  $B_2\sp{(1)}=C_2\sp{(1)}$ type affine Lie algebra---in level one case I. Siladi\' c enumerated leading terms by using a computer.
G. Trup\v cevi\' c in \cite{T} first encountered the ``combinatorics of cascades'' in his construction of
combinatorial bases of the Feigin-Stoyanovsky type subspaces of standard modules for affine Lie algebras
$\widehat{\mathfrak{sl}}_n$, a very special case of which are the ``admissible configurations'' in \cite{FJLMM} and \cite{P1}. 
In \cite{BPT} combinatorial bases of the Feigin-Stoyanovsky type subspaces of all standard modules for affine Lie algebras of type $C_\ell\sp{(1)}$ were constructed. Since the combinatorial parametrization of the leading terms in  \cite{BPT} formally coincides with the one described
above for $n=2\ell$, we feel that this formal similarity  might also  support our conjecture 
(cf. \cite{P2}).


In the last section we formulate conjectured colored Rogers-Ramanujan type combinatorial identities.

We thank Arne Meurman for many stimulating discussions and help in understanding the combinatorics of leading terms.

\section{Vertex algebras for affine Lie algebras}

Let ${\mathfrak g}$ be a simple complex Lie algebra, $\mathfrak h$
a Cartan subalgebra of ${\mathfrak g}$ and $\langle \ , \ \rangle$
a symmetric invariant bilinear form on ${\mathfrak g}$. Via this
form we identify $\mathfrak h$ and $\mathfrak h^*$ and we assume
that $\langle \theta , \theta \rangle=2$ for the maximal root
$\theta$ (with respect to some fixed basis of the root system).
We fix a root vector $x_\theta$ in $\mathfrak g$. Set
$$
\hat{\mathfrak g} =\coprod_{j\in\mathbb Z}{\mathfrak g}\otimes
t^{j}+\mathbb C c, \qquad \tilde{\mathfrak g}=\hat{\mathfrak
g}+\mathbb C d.
$$
Then $\tilde{\mathfrak g}$ is the associated untwisted affine
Kac-Moody Lie algebra (cf. \cite{K}) with the commutator
$$
[x(i),y(j)]=[x,y](i+j)+i\delta_{i+j,0}\langle x,y\rangle c.
$$
Here, as usual, $x(i)=x\otimes t^{i}$ for $x\in{\mathfrak g}$ and
$i\in\mathbb Z$, $c$ is the canonical central element, and
$[d,x(i)]=ix(i)$.  We identify ${\mathfrak g}$ and
${\mathfrak g}\otimes 1$. Set
$$
\tilde{\mathfrak g}_{<0} =\coprod_{j<0}{\mathfrak g}\otimes
t^{j},\qquad \tilde{\mathfrak g}_{\leq 0} =\coprod_{j\leq
0}{\mathfrak g}\otimes t^{j}+\mathbb C d,\qquad\tilde{\mathfrak
g}_{\geq 0} =\coprod_{j\geq 0}{\mathfrak g}\otimes t^{j}+\mathbb C
d.
$$

For $k\in\mathbb C$ denote by $\mathbb C v_k$ the one-dimensional
$(\tilde{\mathfrak g}_{\geq 0}+\mathbb C c)$-module on which
$\tilde{\mathfrak g}_{\geq 0}$ acts trivially and $c$ as the
multiplication by $k$. The affine Lie algebra $\tilde{\mathfrak
g}$ gives rise to the vertex operator algebra (see \cite{FLM} and
\cite{FHL}, here we use the notation from \cite{MP3})
$$
N(k\Lambda_0)=U(\tilde{\mathfrak
g})\otimes_{U(\tilde{\mathfrak g}_{\geq 0}+\mathbb C c)}\mathbb C
v_k
$$
for level $k\neq -g^\vee$, where $g^\vee$ is the dual Coxeter
number of ${\mathfrak g}$; it is generated by the fields
\begin{equation}\label{E:3.1}
x(z)=\sum_{m\in\mathbb Z}x_mz^{-m-1},\qquad x\in{\mathfrak g},
\end{equation}
where we set $x_m=x(m)$ for $x\in{\mathfrak g}$.
As usual, for $v\in N(k\Lambda_0)$ we denote the associated vertex operator by $Y(v,z)=\sum_{m\in\mathbb Z }v_mz^{-m-1}$,
and the vacuum vector by $\mathbf  1$. By the
state-field correspondence we have
$$
x(z)=Y(x(-1)\mathbf 1 ,z)\quad\text{for}\ x\in{\mathfrak g}.
$$
The $\mathbb Z$-grading is given by $L_0=-d$.
>From now on we fix the level $k\in\mathbb Z_{>0}$.

\section{Annihilating fields of standard modules}

For the fixed positive integer level $k$ the generalized Verma
$\tilde{\mathfrak g}$-module $N(k\Lambda_0)$ is  reducible, and we
denote by $N^1(k\Lambda_0)$ its maximal $\tilde{\mathfrak
g}$-submodule. By \cite[Corollary 10.4]{K} the submodule
$N^1(k\Lambda_0)$ is generated by the singular vector
$x_\theta(-1)^{k+1}\mathbf 1$. Set
$$
R=U(\mathfrak g)x_\theta(-1)^{k+1}\mathbf 1,\qquad \bar R =
\mathbb C\text{-span}\{r_m \mid r \in R, m \in \mathbb Z\}.
$$
Then $R\subset N^1(k\Lambda_0)$ is an irreducible $\mathfrak
g$-module, and $\bar R$ is the corresponding loop
$\tilde{{\mathfrak g}}$-module for the adjoint action.
We have the following theorem (see \cite{DL}, \cite{FZ}, \cite{Li},
\cite{MP3}):
\begin{theorem}\label{T:5.1}
 Let $M$ be a highest weight $\tilde{\mathfrak
g}$-module of level $k$. The following are equivalent:
\begin{enumerate}
\item $M$ is a standard module,
\item $\bar R$ annihilates $M$.
\end{enumerate}
\end{theorem}
\noindent
This theorem implies that for a dominant integral weight
${\Lambda}$ of level ${\Lambda}(c) = k$ we have
$$
\bar RM({\Lambda}) = M^1({\Lambda}),
$$
where $M^1({\Lambda})$ denotes the maximal submodule of the Verma
$\tilde{\mathfrak g}$-module $M({\Lambda})$. Furthermore, since
$R$ generates the vertex algebra ideal $N^1(k\Lambda_0)\subset N(k\Lambda_0)$,
the vertex operators $Y(v,z)$, $v\in N^1(k\Lambda_0)$, annihilate all
standard $\tilde{\mathfrak g}$-modules
$$
L({\Lambda})=M({\Lambda})/M^1({\Lambda})
$$
of level $k$. We shall call the elements $r_m\in\bar R$ {\it relations} (for
standard modules), and $Y(v,z)$, $v\in N^1(k\Lambda_0)$, {\it
annihilating fields} (of standard modules). The
field
$$
Y(x_\theta(-1)^{k+1}\mathbf 1,z)=x_\theta(z)^{k+1}
$$
generates all annihilating fields. We also write
$$
Y(x_\theta(-1)^{k+1}\mathbf 1,z)=\sum_{m\in\mathbb Z }r_{(k+1)\theta}(m)z^{-m-k-1}.
$$

\section{Leading terms}

Set $\bar{\mathfrak g}=\hat{\mathfrak g}/\mathbb C c$. The associative algebra $\mathcal U=U(\hat{\mathfrak g})/(c-k)$
inherits {}from $U(\hat{\mathfrak g})$ the filtration $\mathcal
U_\ell$, $\ell\in\mathbb Z_{\geq 0}$; let us denote by $\mathcal
S\cong S(\bar{\mathfrak g})$ the corresponding commutative graded
algebra. Let $B$ be a basis of ${\mathfrak g}$.
We fix the basis $\bar B$ of $\bar{\mathfrak g}$,
$$
 \bar B=\bigcup_{j\in\mathbb Z}
B\otimes t^j.
$$
Let $\preceq$ be a linear order on $\bar B$ such that
$$
i<j\quad\text{implies}\quad x(i)\prec y(j).
$$
The symmetric algebra $\mathcal S$ has a basis $\mathcal P$
consisting of monomials in basis elements $\bar B$. Elements
$\Pi\in\mathcal P$ are finite products of the form
$$
\Pi=\prod_{i=1}^\ell X_i(j_i), \quad X_i(j_i)\in\bar B,
$$
and we shall say that $\Pi$ is a colored partition of degree
$\left|\Pi\right|=\sum_{i=1}^\ell j_i\in \mathbb Z$ and length $\ell
\left(\Pi\right)=\ell$, with parts $X_i(j_i)$ of degree $j_i$ and color
$X_i$. We shall usually assume that parts of $\Pi$ are indexed so
that
\begin{equation}\label{E: colored partition 1}
X_1(j_1)\preceq X_2(j_2)\preceq \dots\preceq X_\ell(j_\ell).
\end{equation}
We associate with a colored partition $\Pi$  its shape
$\text{sh\,}\Pi$, the ``plain'' partition
$$
j_1\leq j_2\leq \dots\leq j_\ell.
$$
The basis element $1\in\mathcal P$ we call the colored partition
of degree 0 and length 0. Note that $\mathcal P\subset\mathcal S$ is a monoid with the unit
element 1, the product of monomials $\Phi$ and $\Psi$ is denoted
by $\Phi\Psi$. For colored partitions $\Phi$, $\Psi$ and
$\Pi=\Phi\Psi$ we shall write $\Phi=\Pi/\Psi$ and
$\Psi\subset\Pi$. We shall say that $\Psi\subset\Pi$ is an
embedding (of $\Psi$ in $\Pi$), notation suggesting that $\Pi$
``contains'' all the parts of $\Psi$.

The set of all colored partitions of
degree $m$ and length $\ell$ is denoted as $\mathcal P^\ell(m)$.
The set of all colored partitions with parts $X_i(j_i)$ of degree
$j_i<0$ is denoted as $\mathcal P_{<0}$.
We shall fix a monomial basis
$$
X(\Pi)=X_1(j_1)X_2(j_2) \dots X_\ell(j_\ell), \quad \Pi\in\mathcal P,
$$
of the enveloping algebra $\mathcal U$ such that (\ref{E: colored partition 1}) holds.
Then, by Poincar\' e-Birkhoff-Witt theorem, we have a basis
\begin{equation}\label{E: colored partition 2}
X(\Pi)\,\mathbf 1, \qquad \Pi\in\mathcal P_{<0},
\end{equation}
of $N(k\Lambda_0)$, and on the quotient $L(k\Lambda_0)$ a PBW spanning set of the form (\ref{E: colored partition 2}).

Clearly $\bar B\subset\mathcal P$, viewed as colored partitions of
length 1. We assume that on $\mathcal P$ we have a linear order
$\preceq$ which extends the order $\preceq$ on $\bar B$. Moreover,
we assume that order $\preceq$ on $\mathcal P$ has the following
properties:
\begin{itemize}
\item  $\ell(\Pi)>\ell(\Phi)$ implies $\Pi\prec\Phi$.
\item  $\ell(\Pi)=\ell(\Phi)$, $|\Pi|<|\Phi|$ implies $\Pi\prec\Phi$.
\item  Let $\ell(\Pi)=\ell(\Phi)$, $|\Pi|=|\Phi|$. Let $\Pi$ be a partition
$b_1(j_1)\preceq b_2(j_2)\preceq \dots\preceq b_\ell(j_\ell)$ and
$\Phi$ a partition $a_1(i_1)\preceq a_2(i_2)\preceq \dots\preceq
a_\ell(i_\ell)$. Then $\Pi\preceq \Phi$ implies $j_\ell\leq
i_\ell$.
\item  Let $\ell \ge 0$, $m \in  \mathbb Z$ and let
$S \subset \mathcal P$ be a nonempty subset such that all $\Pi$ in
$S$ have length $\ell(\Pi) \le \ell$ and degree $\vert \Pi \vert =
m$. Then $S$ has a minimal element.
\item   $\Phi  \preceq  \Psi$ implies
$\Pi\Phi  \preceq \Pi\Psi$.
\item  The relation $\Pi\prec\Phi$ is a well order on $\mathcal P_{< 0}$.
\end{itemize}

\begin{remark}
An order with these properties is used in \cite{MP3}; colored
partitions are compared first by length and degree, and then by
comparing degrees of parts and colors of parts in the reverse
lexicographical order. In this paper we shall use the same order on $\mathcal P$
extending a chosen linear order on $B$.
\end{remark}

\begin{remark}\label{Remark42}
Note that for elements $X_1(j_1), X_2(j_2), \dots, X_\ell(j_\ell)\in\bar B$ and any permutation $\sigma$ we have
\begin{equation}\label{E: colored partition 2.5}
  X_1(j_1)X_2(j_2) \dots X_\ell(j_\ell)- X_{\sigma(1)}(j_{\sigma(1)})X_{\sigma(2)}(j_{\sigma(2)}) \dots
X_{\sigma(\ell)}(j_{\sigma(\ell)})\in\mathcal U_{\ell-1}.
\end{equation}
So if (\ref{E: colored partition 1}) holds and $\Pi= X_1(j_1)\dots X_\ell(j_\ell)$, our first requirement on order $\succeq$ and
(\ref{E: colored partition 2.5}) imply
\begin{equation}\label{E: colored partition 3}
X_{\sigma(1)}(j_{\sigma(1)}) \dots X_{\sigma(\ell)}(j_{\sigma(\ell)})=X(\Pi)+\sum_{\Phi\succ\Pi}c_\Phi X(\Phi)
\end{equation}
for  some coefficients $c_\Phi$ for $\Phi\in\mathcal P$.
Since we are mostly interested in leading terms of relations, due
to (\ref{E: colored partition 3}) we shall often make no distinction between $\Pi\in\mathcal S$ and $X(\Pi)\in\mathcal U$.
\end{remark}

Relation $r_{(k+1)\theta}(m)$, a coefficient of annihilating field ${x_{\theta}(z)}^{k+1}$, is an infinite sum
\begin{equation}\label{E: leading term 1}
  r_{(k+1)\theta}(m) =\sum_{j_1+\cdots +j_{k+1}=m}x_{\theta}(j_1)\cdots
x_{\theta}(j_{k+1}),
\end{equation}
and the smallest summand in this sum is proportional to
\begin{equation}\label{E: leading term 1.5}
 x_{\theta}(-j-1)^b  x_{\theta}(-j)^a
\end{equation}
for $a+b=k+1$ and $(-j-1)b +(-j)a=m$. Moreover, the shape of every other term $\Phi$ which appears in the sum is greater than the shape
$(-j-1)\sp b  (-j)\sp a$, so we can write
\begin{equation}\label{E: leading term 2}
  r_{(k+1)\theta}(m) =c\, x_{\theta}(-j-1)^b  x_{\theta}(-j)^a+\sum_{ \text{sh\,}\Phi\succ(-j-1)\sp b  (-j)\sp a}c_\Phi X( \Phi)
\end{equation}
for some $c\neq0$ and coefficients $c_\Phi$ for $\Phi\in\mathcal P\sp{k+1}(m)$. The adjoint action of $U(\mathfrak g)$ on $r_{(k+1)\theta}(m) $,
$m\in\mathbb Z$, gives all other relations in $\bar R$. For $u\in U(\mathfrak g)$ the relation
$ r(m)=u\cdot  r_{(k+1)\theta}(m) $ can be written as
\begin{equation}\label{E: leading term 4}
r(m) =\sum_{\text{sh\,}\Psi= (-j-1)\sp b  (-j)\sp a}c_\Psi   X(\Psi)+\sum_{\text{sh\,}\Psi \succ(-j-1)\sp b  (-j)\sp a}c_\Psi  X( \Psi)
+\sum_{\ell(\Psi)<k+1}c_\Psi  X( \Psi).
\end{equation}

Let $c$ be as in (\ref{E: leading term 2}). The actions of $u\in U(\mathfrak g)$ in $\mathfrak g$-modules $\mathcal U$ and $\mathcal S$ are different, but because of
(\ref{E: colored partition 3}) we have
\begin{equation*}
u\Big(c\, x_{\theta}(-j-1)^b  x_{\theta}(-j)^a\Big)=\sum_{\text{sh\,}\Psi =(-j-1)\sp b  (-j)\sp a}c_\Psi   \Psi
\end{equation*}
with the same coefficients $c_\Psi$ as in the first summand in (\ref{E: leading term 4}). Hence $r(m)\neq0$ if and only if $c_\Psi\neq0$ for some $\Psi$.
The smallest $\Psi\in\mathcal P\sp{k+1}(m)$ which appears in the first sum in (\ref{E: leading term 4}) with $c_\Psi \neq0$ we call {\it the leading term of relation}
$r(m)$ and we denote it as $\ell\!\it t\,  r(m)$. Hence we can rewrite (\ref{E: leading term 4}) as
\begin{equation}\label{E: leading term 6}
r(m) =c_\Phi   X(\Phi)+\sum_{\Psi \succ\Phi}c_\Psi  X( \Psi), \qquad \Phi=\ell\!\it t\,  r(m).
\end{equation}

Set
\begin{equation*}
\mathcal R=\{\ell\!\it t\,  r(m)\mid r\in R, \,  m\in\mathbb Z\},\quad \mathcal D=\mathcal P\setminus \{\Psi\Phi\mid\Psi\in\mathcal P, \Phi\in\mathcal R\}.
\end{equation*}
In other words, $\mathcal D$ is the set of all colored partitions which {\it do not contain} any leading term from $\mathcal R$---we say that colored
partitions in $\mathcal D$ satisfy {\ difference conditions $\mathcal R$}.
Since  $r(m)=0$ on $L(k\Lambda_0)$, by using  (\ref{E: leading term 6}) we can replace monomial $X(\Phi)$ with a combination of monomials $X(\Psi)$, $\Psi \succ\Phi$,
and reduce the spanning set (\ref{E: colored partition 2}) to a smaller spanning set parametrized with colored partitions satisfying difference conditions $\mathcal R$:

\begin{proposition}\label{P: spanning set of standard module}  The standard module $L(k\Lambda_0)$ is spanned by the set of vectors
\begin{equation}\label{E:  leading term 8}
X(\Pi)\,\mathbf 1, \qquad \Pi\in\mathcal D\cap\mathcal D_{<0}.
\end{equation}
\end{proposition}

\begin{remark}
In spite of the fact that the spanning set (\ref{E:  leading term 8}) is obtained by using all defining relations $\bar R$ for level $k$ standard
modules, this set need not be a basis of $L(k\Lambda_0)$! By results in \cite{FKLMM} and \cite{MP3} it is a basis for $\widehat{\mathfrak{sl}}_2$, but by \cite{MP4}
for $\widehat{\mathfrak{sl}}_3$ it is not---at least it is not a basis for a chosen ordered basis $B$ of ${\mathfrak{sl}}_3$.  And then again, by \cite{S}, it
is a basis for the basic modules for affine Lie algebras of types $A_2\sp{(2)}$ and $B_2\sp{(1)}$. In \cite{MP4} it is shown how it can happen that  (\ref{E:  leading term 8}) is not a basis, but it is not clear ``why" it happens---one way or the other.
\end{remark}

\section{Simple Lie algebra of type $C_n$}

We fix a simple Lie algebra $\mathfrak{g}$ of type $C_n$, $n\geq 2$. For a given Cartan subalgebra $\mathfrak h$ and the corresponding
root system $\Delta$ we can write
\begin{equation*}
\Delta = \{\pm(\varepsilon_i\pm\varepsilon_j) \mid i,j=1,...,n\}\setminus\{0\}\ .
\end{equation*}
We chose simple roots as in \cite{B}
\begin{equation*}
\alpha_1= \varepsilon_1-\varepsilon_2,  \ \alpha_2=\varepsilon_2-\varepsilon_3, \  \cdots \ \alpha_{n-1}=\varepsilon_{n-1}-\varepsilon_{n},
 \  \alpha_n = 2\varepsilon_n.
\end{equation*}
Then $\theta=2 \varepsilon_1$. For each $\alpha\in\Delta$ we chose a root vector $X_{\alpha}$ such that $[X_{\alpha},X_{-\alpha}]=\alpha^{\vee}$. For root vectors
$X_{\alpha}$ we shall use the following notation:
$$\begin{array}{ccc}
X_{ij}\quad \text{or just}\quad ij &  \text{if}\   &\alpha =\varepsilon_i + \varepsilon_j\ , \ i\leq j\,,\\
X_{\underline{i}\underline{j}}\quad \text{or just}\quad \underline{i}\underline{j} & \ \text{if}\  &\alpha =-\varepsilon_i - \varepsilon_j\ , \ i\geq j\,,\\
X_{i\underline{j}}\quad \text{or just}\quad i \underline{j} & \ \text{if}\  &\alpha =\varepsilon_i - \varepsilon_j\ , \ i\neq j\,.\\
\end{array}
$$
With previous notation $x_\theta=X_{11}$. We also write for $i=1, \dots, n$
$$
X_{i\underline{i}}=\alpha_i^{\vee}\ \text{or just}\ i\underline{i} \,.
$$
These vectors $X_{ab}$ form a basis $B$ of $\mathfrak g$ which we shall write in a triangular scheme. For example, for $n=3$ the basis $B$ is
$$\begin{array}{cccccc}
11 &  &&  & & \\
12 & 22 & & & & \\
13 & 23 & 33 & & & \\
1\underline{3} & 2\underline{3} & 3\underline{3} & \underline{3}\underline{3}  & & \\
1\underline{2} & 2\underline{2} & 3\underline{2} & \underline{3}\underline{2}  & \underline{2}\underline{2}& \\
1\underline{1} & 2\underline{1} & 3\underline{1} & \underline{3}\underline{1}  & \underline{2}\underline{1} & \underline{1}\underline{1}.
\end{array}
$$
In general for the set of indices $\{1,2,\cdots ,n,\underline{n},\cdots ,\underline{2},\underline{1}\}$ we use order
\begin{equation*}
1\succ 2\succ\cdots\succ n-1\succ n\succ \underline{n} \succ  \underline{n-1} \succ\cdots \succ \underline{2} \succ \underline{1}
\end{equation*}
and a basis element $X_{ab}$ we write in $a^{th}$ column and $b^{th}$ row,
\begin{equation}\label{E:12.6}
B=\{X_{ab}\mid b\in\{1,2,\cdots ,n,\underline{n},\cdots ,\underline{2},\underline{1}\},\ a\in\{1,\cdots ,b\}\} .
\end{equation}
By using (\ref{E:12.6}) we define on the basis $B$ the corresponding reverse  lexicographical order, i.e.
\begin{equation*}
X_{ab}\succ X_{a'b'} \ if\ b\succ b'\ or \ b=b' \ and\ a\succ a'\ .
\end{equation*}
In other words, $X_{ab}$ is larger than $X_{a' b'}$ if $X_{a' b'}$ lies in a row $b'$ below the row $b$, or $X_{ab}$ and
$X_{a' b'}$ are in the same row $b=b'$, but $X_{a'b'}$  is to the right of $X_{ab}$.

For $r\in \{1,\cdots , n,\underline{n}, \cdots , \underline{1}\}$ we introduce the notation
$$
\triangle_r \quad\text{and} \quad {}\sp r\!\triangle
$$
for triangles in the basis $B$ consisting of rows $\{1,\dots,r\}$ and columns $\{r,\dots,\underline{1}\}$. These two triangles have vertices $11,1r, rr$ and
$rr, r\underline{1}, \underline{1}\underline{1}$ with a common vertex $rr$. Note that $\triangle_r$ is above the vertex $rr$, and $ {}\sp r\!\triangle$
is bellow it.  For example, for $n=3$ and $r=\underline{3}$ we have triangles $\triangle_{\underline{3}}$ and $ {}\sp {\underline{3}}\triangle$
$$\begin{array}{cccccc}
11 &  &&  & & \\
12 & 22 & & & & \\
13 & 23 & 33 & & & \\
1\underline{3} & 2\underline{3} & 3\underline{3} & \underline{3}\underline{3}  & & \\
&& & \underline{3}\underline{2}  & \underline{2}\underline{2}& \\
&&& \underline{3}\underline{1}  & \underline{2}\underline{1} & \underline{1}\underline{1}.
\end{array}
$$
\bigskip

We say that $[rs]=\text{ad}\, X_{rs}$ is an arrow on the basis $B$. Let $X_\alpha$ and $X_\beta$ be two elements in the basis $B$ and let
\begin{equation}\label{E: arrows 4}
[\alpha]X_\beta=[ X_{\alpha}, X_\beta]=\sum_{\gamma\in B}c_{\alpha\beta\gamma}X_\gamma.
\end{equation}
Than we say that the arrow $[\alpha]$ {\it moves a point $\beta$ to a point $\gamma$} if $c_{\alpha\beta\gamma}\neq0$ and we write
$$
\beta\overset{[\alpha]}{\longrightarrow}\gamma \qquad\text{or simply}\qquad \beta{\longrightarrow}\gamma .
$$
We say that the arrow
$[\alpha]$ {\it does not move a point $\beta$} if $[\alpha]X_\beta=0$.
\bigskip

For example, for $n=3$ and the arrow $[2\underline{1}]$ we have $11\longrightarrow12$ and $1r\longrightarrow 2r$ for $r\neq 1$, and
$s\underline{2}\longrightarrow s\underline{1}$ for $s\neq\underline{1}$ and
$\underline{2}\underline{1}\longrightarrow\underline{1}\underline{1}$, or we may also write
$$
11\overset{[2\underline{1}]}{\longrightarrow} 12 \overset{[2\underline{1}]}{\longrightarrow}22\quad\text{and}
\quad 13\overset{[2\underline{1}]}{\longrightarrow}23,
\quad 1\underline{3}\overset{[2\underline{1}]}{\longrightarrow}2\underline{3},
\quad 1\underline{2}\overset{[2\underline{1}]}{\longrightarrow}2\underline{2},
\quad 1\underline{1}\overset{[2\underline{1}]}{\longrightarrow}2\underline{1},
$$
$$
 1\underline{2}\overset{[2\underline{1}]}{\longrightarrow}1\underline{1},\quad
2\underline{2}\overset{[2\underline{1}]}{\longrightarrow}2\underline{1},\quad
3\underline{2}\overset{[2\underline{1}]}{\longrightarrow}3\underline{1},\quad
\underline{3}\underline{2}\overset{[2\underline{1}]}{\longrightarrow}\underline{3}\underline{1}\quad\text{and}\quad
\underline{2}\underline{2}\overset{[2\underline{1}]}{\longrightarrow} \underline{2}\underline{1}
 \overset{[1\underline{1}]}{\longrightarrow}\underline{1}\underline{1}.
$$
\bigskip

The arrow  $[\alpha]=\text{ad}\, X_{\alpha}$ acts as a derivation on the symmetric algebra $S(\mathfrak g)$, so for
monomials with factors  $X_\delta\in B$ we have
\begin{equation*}
[\alpha]\prod_{\delta\in B}X_\delta\sp{m_\delta}=\sum_{\beta\in B}m_\beta\prod_{\delta\neq\beta}X_\delta\sp{m_\delta}
X_\beta\sp{m_\beta-1}\left([\alpha]X_\beta\right)
\end{equation*}
and, after inserting (\ref{E: arrows 4}),
\begin{equation}\label{E: arrows 6}
[\alpha]\prod_{\delta\in B}X_\delta\sp{m_\delta}
=\sum_{\beta, \gamma\in B}m_\beta c_{\alpha\beta\gamma}\prod_{\delta\neq\beta}X_\delta\sp{m_\delta}
X_\beta\sp{m_\beta-1}X_\gamma.
\end{equation}
In our (combinatorial) arguments based on this formula it will be convenient to visualize the monomial
$$
\prod_{\delta\in B}X_\delta\sp{m_\delta}
$$
as the set of points at places $\delta$ in our triangle $B$, with multiplicities $m_\delta$, and the resulting monomials in  (\ref{E: arrows 6})
$$
\prod_{\delta\neq\beta}X_\delta\sp{m_\delta}
X_\beta\sp{m_\beta-1}X_\gamma
$$
for $c_{\alpha\beta\gamma}\neq0$ we visualize as a set of points in the basis $B$ obtained by moving one point from the place $\beta$ to the place $\gamma$,
thus changing the multiplicities $m_{\beta}\mapsto m_{\beta}-1$ and $m_{\gamma}\mapsto m_{\gamma}+1$.

We prove the following lemmas by checking when $\alpha, \beta\in\Delta$ implies $\alpha+\beta\in\Delta$.
\begin{lemma}
\label{L: arrows 1} Let $r\in\{2,\dots, n\}$. Arrow $[r\underline{1}]$ moves
\begin{enumerate}
\item $1s\longrightarrow sr$ \ for $s=1,\dots,r$, and
\item does not move any point in $\triangle_r\setminus\{11, 12,\dots, 1r\}$.
\end{enumerate}
\end{lemma}

\begin{lemma}
\label{L: arrows 2} Let $r\in\{3,\dots, n\}$. Arrows $ [2\underline{1}], [3\underline{1}], \dots, [r-1,\underline{1}]$
\begin{enumerate}
\item move $11$ to points $ 12, 13, \dots, 1(r-1)$,
\item  move points $12, 13, \dots, 1(r-1)$ into $\triangle_{r-1}$,
\item move $1r$ to $sr$ for $s=2,\dots, r-1$, and
\item  do not move any point in $\triangle_r\setminus\{11, 12,\dots, 1r\}$.
\end{enumerate}
\end{lemma}

We prove the following lemma by expressing a dual root $\alpha\sp\vee$ in
terms of simple coroots $\alpha_1\sp\vee$, \dots, $\alpha_n\sp\vee$\,:
\begin{lemma}
\label{L: arrows 3} Let $r\in\{2,\dots, n\}$. Then
\begin{enumerate}
\item $ [\underline{r}\underline{1}]X_{1r}=-X_{1\underline{1}}\dots-X_{r-1\underline{r-1}}-2X_{r\underline{r}}\dots-2X_{n\underline{n}}$ ,
\item $ [\underline{r}\underline{r}]X_{rr}=-X_{r\underline{r}}\dots-X_{n\underline{n}}$.
\end{enumerate}
\end{lemma}

\begin{lemma}
\label{L: arrows 4} Let $r\in\{2,\dots, n\}$. Arrow $[\underline{r}\underline{1}]$ moves
\begin{enumerate}
\item $1s\longrightarrow s\underline{r}$ \ for $s=1,\dots,\underline{r}$, $s\neq r$,
\item $ir\longrightarrow i\underline{1}$ \ for $i=2,\dots, r-1, r$,
\item $ri\longrightarrow i\underline{1}$ \ for $i=r,\dots, \underline{r}$,
\item apart  from $1r$, arrow $[\underline{r}\underline{1}]$  does not move any other point in $\triangle_{\underline{r}}$.
\end{enumerate}
\end{lemma}

For $s\in\{\underline{1},\dots, \underline{n}\}$ and $t\in\{1,\dots, n\}$ such that $s=\underline{t}$ we write $t=\underline{s}$.

\begin{lemma}
\label{L: arrows 5} Let $r\in\{1,\dots, n\}$. Arrow $[s\underline{1}]$, $s\in\{2,\dots, \underline{r+1}\}$, moves
\begin{enumerate}
\item $11\longrightarrow 1s$,
\item $1p\longrightarrow ps$ \ for $p=2,\dots, s$, $p\neq \underline{s}$,
\item $1p\longrightarrow sp$ \ for $p=s,\dots, \underline{r+1}$, $p\neq \underline{s}$,
\item $i\underline{s}\longrightarrow i\underline{1}$ \ for $i=2,\dots, \underline{s}$,
\item $\underline{s}i\longrightarrow i\underline{1}$ \ for $i=\underline{s},\dots, \underline{r+1}$,
\item apart  from $1\underline{s}$, arrow $[s\underline{1}]$  does not move any other point in $\triangle_{\underline{r}}$.
\end{enumerate}
\end{lemma}

\begin{lemma}
\label{L: arrows 6} Let $r\in\{2,\dots, n\}$. Arrow $[\underline{r}\underline{r}]$ moves
\begin{enumerate}
\item $ir\longrightarrow i\underline{r}$ \ for $i=1,\dots, r-1$,
\item $ri\longrightarrow i\underline{r}$ \ for $i=r+1,\dots, \underline{r}$,
\item apart  from $rr$, arrow $[\underline{r}\underline{r}]$  does not move any other point in $\triangle_{\underline{r}}$.
\end{enumerate}
\end{lemma}
Note that for simple roots we have arrows
$$
[2\underline{1}]=\text{ad}\,X_{-\alpha_1}, \quad \dots, \quad  [n,\underline{n-1}]=\text{ad}\,X_{-\alpha_{n-1}},
\quad  [\underline{n}\underline{n}]=\text{ad}\,X_{-\alpha_n}\,.
$$

\begin{lemma}
\label{L: arrows 7} Let $r\in\{1,\dots, n-1\}$. Arrow $[r+1,\underline{r}]$ moves
\begin{enumerate}
\item $ir\longrightarrow i(r+1)$ \ for $i=1,\dots, r$,
\item $ri\longrightarrow (r+1)i$ \ for $i=r+1,\dots, \underline{1}$,  $i\neq \underline{r+1}$
\item $i \underline{r+1}\longrightarrow  i\underline{r}$ \ for $i=1,\dots, \underline{r+1}$,
\item $\underline{r+1}i\longrightarrow  i\underline{r}$ \ for $i=\underline{r},\dots, \underline{1}$.
\end{enumerate}
\end{lemma}

\begin{lemma}
\label{L: arrows 8}  Arrow $[\underline{n},\underline{n}]$ moves
\begin{enumerate}
\item $in\longrightarrow i\underline{n}$ \ for $i=1,\dots, n$,
\item $ni\longrightarrow \underline{n}i$ \ for $i=\underline{n},\dots, \underline{1}$.
\end{enumerate}
\end{lemma}


\section{Leading terms of relations for $C_n\sp{(1)}$}

With order $\preceq$ on the basis $B$ we  define a linear order on $\bar B=\{X(j)\mid X\in B, j\in\mathbb Z\}$ by
$$
X_\alpha(i)\prec X_\beta(j)\quad\text{if}\quad i<j \quad\text{or}\quad i=j, \ X_\alpha\prec X_\beta.
$$
With order $\preceq$ on $\bar B$ we  define a linear order on  $\mathcal{P}$ by
$$
\Pi\prec\Phi\quad\text{if}
$$
\begin{itemize}
\item  $\ell(\Pi)>\ell(\Phi)$ or
\item  $\ell(\Pi)=\ell(\Phi)$, $|\Pi|<|\Phi|$ or
\item  $\ell(\Pi)=\ell(\Phi)$, $|\Pi|=|\Phi|$, $\text{sh\,}\Pi\prec\text{sh\,}\Psi$ in the  reverse lexicographical order or
\item  $\ell(\Pi)=\ell(\Phi)$, $|\Pi|=|\Phi|$, $\text{sh\,}\Pi=\text{sh\,}\Psi$ and colors of $\Pi$ are smaller than the colors of $\Psi$ in reverse lexicographical order.
\end{itemize}
For example,
\begin{equation*}
X_{11}(-3)^2 X_{1\underline{1}}(-2)^2 X_{11}(-2) \prec X_{\underline{1}\underline{1}}(-3) X_{11}(-3) X_{11}(-2)^3
\end{equation*}
because these two colored partitions have the same shape $(-3)^2(-2)^3$ with colors
\begin{equation*}
11\ 11; 1\underline{1}\ 1\underline{1}\ 11\quad \text{and}\quad  \underline{1}\underline{1}\ 11; 11\ 11\ 11
\end{equation*}
and comparing from the right we se $11 = 11$, $1\underline{1} \prec 11$. Generally, we may visualize  colored partitions
$\Pi, \Psi\in \mathcal{P}_{<0}$ by corresponding Young diagrams. For instance, the above mentioned colored partitions
are presented by
\def\sq{\lower.3ex\vbox{\hrule\hbox{\vrule height1.2ex depth1.2ex\kern2.4ex
\vrule}\hrule}\,}
\begin{eqnarray*}
\begin{aligned}
11\quad&\sq\sq\sq       \\
11\quad&\sq\sq\sq         \\
1\underline{1}\quad&\sq\sq  \\
1\underline{1}\quad&\sq\sq  \\
11\quad&\sq\sq
\end{aligned}
& \prec &
\begin{aligned}
\underline{1}\underline{1}\quad&\sq\sq\sq       \\
11\quad&\sq\sq\sq         \\
11\quad&\sq\sq  \\
11\quad&\sq\sq  \\
11\quad&\sq\sq\ .
\end{aligned}
\end{eqnarray*}
\bigskip

A sequence of basis elements $(X_{a_1b_1}, X_{a_2b_2},\cdots ,X_{a_sb_s})$  is a {\it cascade $\mathcal{C}$
in the basis $B$} if for each $i\in \{1,2,\cdots , s-1\}$ we have  $b_{i+1}\prec b_i$ and $a_{i+1}= a_i$, or $b_{i+1} = b_i$ and $a_{i+1}\succ a_i$.
We can visualize a cascade $\mathcal{C}$ in the basis $B$ as a staircase in the triangle $B$ going downwards from the right to the left, or as
a sequence of waterfalls flowing from the right to the left. Sometimes we shall think of a cascade $\mathcal{C}$  as a
set of points in the basis $B$ and write $\mathcal{C}\subset B$.

We say that $\mathcal{C}$  is a {\it cascade with multiplicities} if
for each $X_{a_i b_i}$ in $\mathcal{C}$ a multiplicity $m_{a_ib_i}\in \mathbb{{Z}}_{\geq 0}$ is given.
By abuse of language, we shall say that  in the cascade  $\mathcal{C}$ with multiplicities $X_{a_i b_i}$  is the {\it place}
$a_i b_i\in\mathcal{C}\subset B$  with $m_{a_ib_i}$  {\it points}. We shall also write a cascade
with multiplicities $\mathcal{C}$ in the basis $B$ as a monomial
\begin{equation*}
\prod_{\alpha\in \mathcal{C}} X_{\alpha}^{m_{\alpha}}.
\end{equation*}
For example,
\begin{equation}\label{E: cascade 2}
X_{33}^5 X_{23}^0 X_{2\underline{3}}^3 X_{2\underline{2}}^1 X_{1\underline{2}}^1
\end{equation}
is a cascade with multiplicities which goes from $33$ one step left, then two steps down and than one step left. Along the way we have $5$ points at
the position $33$,
no points on $23$, $3$ points on $2\underline{3}$, and so on:
$$\begin{array}{cccccc}
\cdot&  &&  & & \\
\cdot& \cdot &&  & & \\
\cdot&  0&5&  & & \\
\cdot&  3&\cdot&  \cdot& & \\
1& 1 &\cdot&\cdot  & \cdot& \\
\cdot&  \cdot&\cdot& \cdot &\cdot & \cdot\\
\end{array}
$$
For $j\in\mathbb Z$ and a cascade (with multiplicities)  $\mathcal C$ we can replace each $X_{a_i b_i}$ in $\mathcal{C}$ with $X_{a_i b_i}(j)$---then
we obtain a sequence
$(X_{a_1b_1}(j), X_{a_2b_2}(j),\cdots ,X_{a_sb_s}(j))$
in $\bar{B}$ which we call {\it a cascade (with multiplicities)  $\mathcal{C}(j)$ at degree $j$}. Sometimes we also denote  a cascade with multiplicities
$\mathcal{C}(j)$ as
\begin{equation*}
\prod_{\alpha\in \mathcal{C}} X_{\alpha}(j)^{m_{\alpha}}.
\end{equation*}
For example, the cascade with multiplicities $\mathcal C$ in (\ref{E: cascade 2}) above gives us a cascade with multiplicities $\mathcal{C}(j)$ at degree $j$:
\begin{equation}
\label{E: cascade 4}
X_{33}(j)^5 X_{23}(j)^0 X_{2\underline{3}}(j)^3 X_{2\underline{2}}(j)^1 X_{1\underline{2}}(j)^1.
\end{equation}
Note that in (\ref{E: cascade 4}) factors are not written in ascending order as in (\ref{E: colored partition 1})---here we prefer a way of writing appropriate for
cascade structure (see Remark \ref{Remark42}).
We say that two cascades are an admissible pair $(\mathcal B,\mathcal A)$ if
\begin{equation*}
\mathcal{B}\subset\triangle_r,\quad\text{and}\quad \mathcal{A}\subset {}\sp r\!\triangle
\end{equation*}
for some $r$. We shall also consider the case when $\mathcal B$ is empty and $ \mathcal{A}\subset {}\sp 1\!\triangle (= B)$.
For general rank we may visualize admissible pair of cascades as figure below
\begin{figure}[h]
\begin{center}
\includegraphics[width=2in]{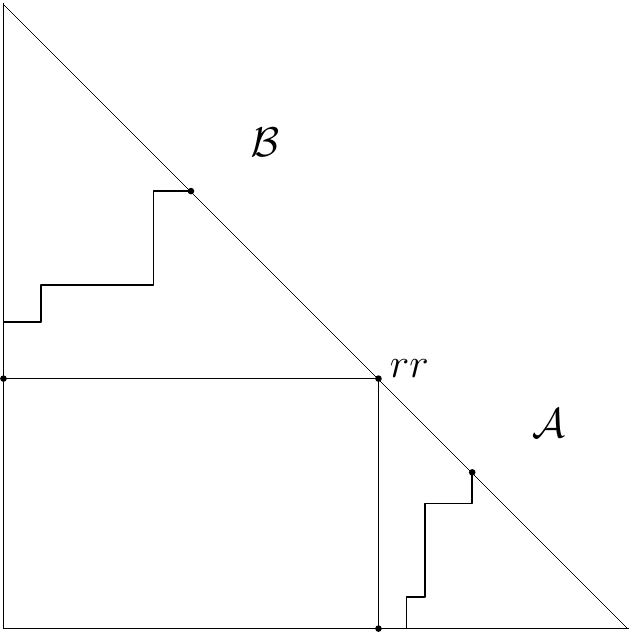}\\
Figure 1.
\end{center}
\end{figure}
\\


\begin{theorem}\label{T: leading terms}
Let
\begin{equation}\label{E:12.shape}
(-j-1)^b (-j)^{a},\quad j\in \mathbb{Z},\quad a+b=k+1,\quad b\geq 0,
\end{equation}
be a fixed shape and let $\mathcal{B}$ and $\mathcal{A}$ be two cascades in degree $-j-1$ and $-j$, with multiplicities
$(m_{\beta,j+1}, \beta\in \mathcal{B})$ and $(m_{\alpha,j}, \alpha \in \mathcal{A})$, such that
\begin{equation}\label{E:12.11}
 \sum_{\beta\in\mathcal{B}}m_{\beta,j+1} =b, \quad \sum_{\alpha\in\mathcal{A}}m_{\alpha,j} =a.
\end{equation}
Let $r\in \{1,\cdots , n,\underline{n}, \cdots , \underline{1}\}$.
If the points of cascade $\mathcal{B}$ lie in the upper triangle $\triangle_r $ and the points of cascade $\mathcal{A}$ lie in the lower triangle
$ {}\sp r\!\triangle$, than
\begin{equation}\label{E:12.12}
\prod_{\beta \in \mathcal{B}} X_{\beta}(-j-1)^{m_{\beta,j+1}}\ \prod_{\alpha \in \mathcal{A}} X_{\alpha}(-j)^{m_{\beta,j}}
\end{equation}
is the leading term of a relation for level $k$ standard module for affine Lie algebra of the type $C_n\sp{(1)}$.
\end{theorem}

Before we prove the theorem  let us make a few remarks.
\begin{remark}
We believe that all leading terms of level $k$ relations $\bar{R}$ are given by (\ref{E:12.12}). In the case $k=1$ and $2$ we can check this by direct
calculation. On one side, by using Weyl's character formula for simple Lie algebra $C_n$, we have
\begin{eqnarray*}
\dim L(2\theta) &=& {2n+3\choose 4},\\
\dim L(3\theta) &=& {2n+5\choose 6} .
\end{eqnarray*}

On the other side, in the case $k=1$  for the shape $(-j)^2$ the number of leading terms (\ref{E:12.12}) is
\begin{equation*}
 \sum_{i_1=1}^{2n} \sum_{j_1=1}^{i_1}\sum_{i_2=i_1}^{2n} \sum_{j_2=1}^{j_1}1 =  {2n+3\choose 4},
\end{equation*}
and for the shape $(-j-1)(-j)$
\begin{equation*}
 \sum_{i_1=1}^{2n} \sum_{j_1=1}^{i_1}\sum_{i_2=i_1}^{2n} \sum_{j_2=i_1}^{i_2}1 ={2n+3\choose 4}.
\end{equation*}
In the case $k=2$ and the shape $(-j)^3$ the number of leading terms (\ref{E:12.12}) is
\begin{equation*}
\sum_{i_1=1}^{2n}\sum_{j_1=1}^{i_1} \sum_{i_2=i_1}^{2n}\sum_{j_2=1}^{j_1}\sum_{i_3=i_2}^{2n}\sum_{j_3=1}^{j_2}1 =  {2n+5\choose 6},
\end{equation*}
for the shape $(-j-1)\sp2(-j)$
\begin{equation*}
 \sum_{i_1=1}^{2n} \sum_{j_1=1}^{i_1}\sum_{i_2=i_1}^{2n} \sum_{j_2=1}^{j_1}\sum_{i_3=i_2}^{2n} \sum_{j_3=i_2}^{i_3}1 ={2n+5\choose 6},
\end{equation*}
and for the shape $(-j-1)(-j)\sp2$
\begin{equation*}
 \sum_{i_1=1}^{2n} \sum_{j_1=1}^{i_1}\sum_{i_2=i_1}^{2n} \sum_{j_2=i_1}^{i_2}\sum_{i_3=i_2}^{2n} \sum_{j_3=i_1}^{j_2}1 ={2n+5\choose 6} \ .
\end{equation*}
\end{remark}
\begin{remark}
 The Lie algebra $\mathfrak{g} =\mathfrak{sl}_2$ may be regarded as of type
 $C_n$ for $n=1$, with the standard basis $B$
\begin{equation*}
 x=x_{11}\ \succ \ h=x_{1\underline{1}}\ \succ \ y=x_{\underline{1} \underline{1}}.
\end{equation*}
The standard basis $B$ can be written as the triangle
$$
\begin{array}{cc}
11 & \\
1\underline{1} & \underline{1} \underline{1}
\end{array}
$$
and Theorem \ref{T: leading terms} applies: for the shape  $(-j-1)^b (-j)^a$, $j\in\mathbb Z$, $a+b= k+1$, all leading terms of relations
for level $k$ standard
$\tilde{\mathfrak g}$-modules are monomials
\begin{equation}\label{E:12.16}
\begin{aligned}
&x(-j-1)^b h(-j)^{a_2} x(-j)^{a_1},\quad a_1+a_2 = a , \\
&x(-j-1)^b y(-j)^{a_2} h(-j)^{a_1},\quad  a_1+a_2 = a ,\\
&h(-j-1)^{b_1} x(-j-1)^{b_2} y(-j)^{a},\quad  b_1+b_2 = b , \\
&y(-j-1)^{b_1} h(-j-1)^{b_2} y(-j)^{a},\quad  b_1+b_2 = b
\end{aligned}
\end{equation}
(see Proposition 6.6.1 in \cite{MP3}).
\end{remark}
\begin{example}
In the case $\mathfrak{g}=\mathfrak{sl}_2$ we obtain all leading terms (\ref{E:12.16}) as
\begin{equation*}
\ell\!\it t\,  \left(    (\text{ad}\,  y)^c ( x(-j-1)^b  x(-j)^a)    \right)
\end{equation*}
for $c=0,1,\cdots ,2a+2b$. For example,
\begin{eqnarray}\label{E:12.21}
&& \ell\!\it t  \Big((\text{ad}\, y)^2  (x(-j-1)^2  x(-j)^3)\Big)  \nonumber\\
&& = \ell\!\it t \Big((\text{ad}\,  y)\big(-3 x(-j-1)^2 h(-j) x(-j)^2 - 2 h(-j-1) x(-j-1) x(-j)^3\big )\Big)\nonumber\\
&& =   \ell\!\it t  \Big(6 x(-j-1)^2 h(-j)^2 x(-j) - 6 x(-j-1)^2 y(-j) x(-j)^2 \nonumber\\
&&\quad  + 12 h(-j-1) x(-j-1) h(-j) x(-j)^2 \nonumber\\
&&\quad   + 2 h(-j-1)^2 x(-j)^3 - 4 y(-j-1) x(-j-1) x(-j)^3\Big)\nonumber\\
&&=    x(-j-1)^2 h(-j)^2 x(-j)\nonumber\ .
\end{eqnarray}
We  visualize this calculations by representing monomials
$$
\Big( y(-j-1)^f h(-j-1)^e x(-j-1)^d\Big)    \Big(y(-j)^c h(-j)^b  x(-j)^a\Big)
$$
as multiple points in two copies of triangle $B$, i.e.
$$
\left(\begin{array}{cc}
d & \\
e & f
\end{array}\right)
\left(\begin{array}{cc}
a & \\
b & c
\end{array}\right),
$$
and by representing the action of $\text{ad}\,y $ with an arrow
which moves the points from the place $x=11$ to the place $h=1\underline{1}$, and from the place $h=1\underline{1}$ to the place
$y=\underline{1}\underline{1}$.
Then we can simplify and visualize previous calculation without explicitly keeping track of coefficients: After applying one arrow on
$$\left(\begin{array}{cc}
2 & \\
0 & 0
\end{array}\right)
\left(
\begin{array}{cc}
3 & \\
0 & 0
\end{array}\right)
$$
we obtain two terms by acting on two different factors
$$
\left(\begin{array}{cc}
2 & \\
0 & 0
\end{array}\right)
\left(
\begin{array}{cc}
2 & \\
1 & 0
\end{array}\right),\quad
\left(\begin{array}{cc}
1 & \\
1 & 0
\end{array}\right)
\left(
\begin{array}{cc}
3 & \\
0 & 0
\end{array}\right),
$$
and the first monomial is smaller  because in our reverse lexicographical order we compare parts from the right to the left
$$
h(-j) \prec x(-j), \  x(-j) = x(-j), \ x(-j) = x(-j).
$$
Now we act with the second arrow on each of these monomials---from the first monomial we obtain three terms
$$
\left(\begin{array}{cc}
2 & \\
0 & 0
\end{array}\right)
\left(
\begin{array}{cc}
1 & \\
2 & 0
\end{array}\right),\quad
\left(\begin{array}{cc}
2 & \\
0 & 0
\end{array}\right)
\left(
\begin{array}{cc}
2 & \\
0 & 1
\end{array}\right),\quad
\left(\begin{array}{cc}
1 & \\
1& 0
\end{array}\right)
\left(
\begin{array}{cc}
2 & \\
1 & 0
\end{array}\right),
$$
and from the second monomial we obtain another three terms
$$\left(\begin{array}{cc}
1 & \\
1 & 0
\end{array}\right)
\left(
\begin{array}{cc}
2 & \\
1 & 0
\end{array}\right),\quad
\left(\begin{array}{cc}
0 & \\
2 & 0
\end{array}\right)
\left(
\begin{array}{cc}
3 & \\
0 & 0
\end{array}\right),\quad
\left(\begin{array}{cc}
1 & \\
0 & 1
\end{array}\right)
\left(
\begin{array}{cc}
3 & \\
0 & 0
\end{array}\right).
$$
The smallest term is the very first monomial $x(-j-1)^2 h(-j)^2 x(-j)$ and we see that our guiding "principle" should be:
{\it Act with an arrow on the largest part of the smallest colored partition to obtain the smallest term}.

In general we obtain all leading terms (\ref{E:12.16}) in the following way: First for $c \leq a$ we act with $c$ arrows on
$$\left(\begin{array}{cc}
b & \\
0 & 0
\end{array}\right)
\left(
\begin{array}{cc}
a& \\
0 & 0
\end{array}\right),
$$
and, acting always on the largest possible part $x(-j)$, we get
$$\left(\begin{array}{cc}
b & \\
0 & 0
\end{array}\right)
\left(
\begin{array}{cc}
a-c& \\
c & 0
\end{array}\right).
$$
With $c=a$ the smallest possible part to act on is $h(-j)$ and with additional $d$ arrows, $d \leq a$, we get
$$\left(\begin{array}{cc}
b & \\
0 & 0
\end{array}\right)
\left(
\begin{array}{cc}
0& \\
a-d & d
\end{array}\right).
$$
Once we have obtained a factor $y(-j)^a$, the smallest part the arrow can act on nontrivially is $x(-j-1)$,  and we proceed
with changing the left triangle.
\end{example}

\begin{example}
Essentially the same procedure used for $\mathfrak g={\mathfrak sl}_2$ can be used for $\mathfrak g$ of type $C_n$, $n\geq 2$, a difference
being in a use of different "kinds" of arrows acting on
$$
X_{11} (-j-1)^b X_{11}(-j)^{a}\, .
$$
For example, by applying $2a$ arrows $[2\underline{1}]$ to the factor $X_{11}(-j)^{a}$ we get $X_{22}(-j)^{a}$. We visualize this step
as moving $a$ points from the place $11$ to the place $12$ by using $a$ arrows $[2\underline{1}]$, and after that by using another $a$ arrows
$[2\underline{1}]$ we move $a$
points from the place $12$ to the place $22$ in the second row (see Lemma \ref{L: arrows 6}).  Any additional $[2\underline{1}]$ arrow will act
trivially on monomial
$X_{22}(-j)^{a}$ and will move points corresponding to the factor $X_{11}(-j-1)^{b}$. So with $2a+b+b'= 2a + (b-b') +2b'$ arrows
we get a relation with the leading term (with a premise that we can show it is really the leading term)
$$\left(\begin{array}{lll}
0 & & \\
{b-b'}  & & {b'}\\
 \cdots & & \\
\end{array}\right)
\left(
\begin{array}{lll}
0 & & \\
0& & a\\
 \cdots & &
\end{array}\right)\,.
$$
By following the above idea and notation,  after acting of $2a+b_1+b_2+2b_3$ times with arrows $[ 3\underline{2}]$, we get a
relation with the leading term
 $$
\left(
\begin{array}{lllll}
0 & & & &\\
0 && 0 &&\\
b_1 && b_2 && b_3\\
 \cdots& & &&
\end{array}\right)
\left(\begin{array}{lllll}
0 & & & &\\
0 && 0 &&\\
0 && 0 && a\\
 \cdots& & &&
\end{array}
\right).
$$
At certain point we could start constructing a cascade on $(-j-1)^b$ part, by using, for example, only $2a+b_1$ arrows
$[4\underline{3}]$ to obtain a relation with the leading term
$$
\left(
\begin{array}{lllllll}
0 & & & & & &\\
0 && 0 &&&&\\
0 && b_2 && b_3 &&\\
b_1 && 0 && 0 && 0\\
\cdots&&&&&&\\
\end{array}\right)
\left(\begin{array}{lllllll}
0 & & & & & &\\
0 && 0 &&&&\\
0 && 0 && 0 &&\\
0 && 0 && 0 && a\\
\cdots&&&&&&\\
\end{array}
\right)
$$
After that we could use $a+a'$, $a' < a)$, arrows $[5\underline{4}]$  to start constructing a cascade on $(-j)^{a}$
part and obtain a relation with the leading term
$$
\left(
\begin{array}{lllllllll}
0 & & & & & & & &\\
0 && 0 &&&&&&\\
0 && b_2 && b_3 &&&&\\
b_1 && 0 && 0 && 0 &&\\
0 && 0 && 0 && 0 && 0\\
\cdots &&&&&&&&\\
\end{array}\right)
\left(\begin{array}{lllllllll}
0 & & & & & & & &\\
0 && 0 &&&&&&\\
0 && 0 && 0 &&&&\\
0 && 0 && 0 && 0 &&\\
0 && 0 && 0 && a-a' && a'\\
\cdots &&&&&&&&\\
\end{array}
\right),
$$
that is, a relation with the leading term
$$
X_{14}(-j-1)^{b_1}X_{33}(-j-1)^{b_3}X_{23}(-j-1)^{b_2}X_{55}(-j)^{a'}X_{45}(-j)^{a-a'}.
$$
This is a basic idea how relations with the leading terms (\ref{E:12.12}) in  Theorem \ref{T: leading terms} can be constructed. However,
there is a slight
difficulty we should mention before writing a formal proof.

Take, for example, $n=3$ and the monomial of the form (\ref{E:12.12})
\begin{equation*}
\Pi = X_{1\underline{1}}(-j)^{n_6} X_{1\underline{2}}(-j)^{n_5} X_{1\underline{3}}(-j)^{n_4} X_{13}(-j)^{n_3} X_{12}(-j)^{n_2}
X_{11}(-j)^{n_1}
\end{equation*}
with $n_2\neq 0$ and $n_5\neq 0$. Observe that we have three kinds of arrows corresponding to simple roots, moving one row to the next one:
\begin{equation*}
11 \overset{[2\underline{1}]}{\longrightarrow} 12 \overset{[3\underline{2}]}{\longrightarrow} 13
\overset{[\underline{3}\underline{3}]}{\longrightarrow} 1\underline{3}
 \overset{[3\underline{2}]}{\longrightarrow} 1\underline{2} \overset{[2\underline{1}]}{\longrightarrow} 1\underline{1}\ .
\end{equation*}
So for  $a=n_1 + \cdots + n_6$ we start with $X_{11}(-j)^a$ and apply correct number of arrows on largest possible parts and get
\begin{eqnarray*}
X_{11}(-j)^a \cdots &\overset{[2\underline{1}]}{\longrightarrow}& X_{12}(-j)^{a-n_1} X_{11} (-j)^{n_1} \cdots\nonumber\\
 &\overset{[3\underline{2}]}{\longrightarrow} &  X_{13}(-j)^{a-n_1-n_2}X_{12}(-j)^{n_2} X_{11} (-j)^{n_1} \cdots\nonumber\\
&\overset{[\underline{3}\underline{3}]}{\longrightarrow} &  X_{1\underline{3}}(-j)^{a-n_1-n_2-n_3}X_{13}(-j)^{n_3}
X_{12}(-j)^{n_2} X_{11} (-j)^{n_1}\cdots\ .
\end{eqnarray*}
But now we are stuck if we are to use only arrows $[2\underline{1}]$, $[3\underline{2}]$ i $[3\underline{3}]$: on one side, for moving
points from $1\underline{3}$ to $1\underline{2}$ we should use arrow $[3\underline{2}]$, but in this way we'll not get the smallest
term unless $n_2=0$. (That is, for $n_2\neq 0$ we'll get a smaller term if we apply arrow $[3\underline{2}]$ to a larger factor $X_{12}(-j)^{n_2}$.)
The way around is to use additional kind of arrows
$$11 \overset{[\underline{1}\underline{1}]}{\longrightarrow} 1\underline{1}\qquad \hbox{and}\qquad 12
\overset{[\underline{2}\underline{2}]}{\longrightarrow} 1\underline{2}\ .
$$
With the use of this additional arrows we can construct a relation with the leading term $\Pi$:
\begin{eqnarray}\label{E:12:33}
X_{11}(-j)^a  &\overset{[2\underline{1}]}{\longrightarrow}& \overset{[\underline{1}\underline{1}]}{\longrightarrow}
 \cdots X_{1\underline{1}}(-j)^{n_6} X_{1 2} (-j)^{a-n_1-n_6}  X_{1 1} (-j)^{n_1} \nonumber\\
 \cdots&\overset{[3\underline{2}]}{\longrightarrow}& \overset{[\underline{2}\underline{2}]}{\longrightarrow}
\quad X_{1\underline{1}}(-j)^{n_6} X_{1 \underline{2}} (-j)^{n_5}X_{13}(-j)^{n_3+n_4}X_{12}(-j)^{n_2} X_{11} (-j)^{n_1} \nonumber\\
 \cdots&\overset{[\underline{3}\underline{3}]}{\longrightarrow}&   X_{1\underline{1}}(-j)^{n_6} X_{1\underline{2}}(-j)^{n_5}
X_{1\underline{3}}(-j)^{n_4} X_{13}(-j)^{n_3} X_{12}(-j)^{n_2} X_{11}(-j)^{n_1}.\nonumber
\end{eqnarray}
Similar problem we encounter if we want to construct a relation with the leading term
$$
X_{1\underline{1}}(-j)^{b_1} X_{2 \underline{1}} (-j)^{b_2} X_{3\underline{1}}(-j)^{b_3}X_{\underline{3}\underline{1}}(-j)^{b_4}
X_{\underline{2}\underline{1}} (-j)^{b_5} X_{\underline{1}\underline{1}} (-j)^{b_6}
$$
"from the previous row" by using arrow $[2\underline{1}]$,
$$
\begin{array}{ccccccccccc}
1\underline{2}&&2\underline{2}&&3\underline{2}&&\underline{3}\underline{2}&&\underline{2}\underline{2}&&\\
\downarrow&&\downarrow& &\downarrow& &\downarrow& &\downarrow&&\\
1\underline{1}&\rightarrow&2\underline{1}&\quad&3\underline{1}&\quad&\underline{3}\underline{1}&\quad&
\underline{2}\underline{1}&\rightarrow&\underline{1}\underline{1}\ ,
\end{array}
$$
because we cannot construct the factor $X_{\underline{1}\underline{1}}^{b_6}$ in a leading term by applying arrow on
$X_{\underline{2}\underline{1}}^{b_5+b_6}$
unless $b_1=0$.

\end{example}

\begin{proof}[Proof of Theorem \ref{T: leading terms}] We construct a relation with the leading term (\ref{E:12.12}) by precisely
defined application
of arrows $[rs]=\text{ad}\, X_{rs}$
on a coefficient of the relation
$$
\text{coeff}_{z\sp m}X_{11}(z)\sp{k+1},\quad m=b(j+1)+aj-k-1.
$$
By previous remarks, the leading term of such relation has the shape (\ref{E:12.shape}) and it is enough to analyse the action of arrows $[rs]$
on the colored partition
$$
Z_0=X_{11}(-j-1)\sp{b}X_{11}(-j)\sp{a}.
$$
We prove the theorem in four steps. Since the leading term (\ref{E:12.12}) reminds us of waterfalls on Plitvice lakes, we title these steps as:
preparation of upper barrier, construction of upper cascades,  preparation of lower barrier and construction of lower cascade.
\bigskip

\noindent{\bf Preparation of upper barrier.} \  Let $t\in \{1,\cdots , n,\underline{n}, \cdots , \underline{1}\}$ be the index of uppermost row
in which the upper cascade $\mathcal B$ in (\ref{E:12.12}) has a point. For $i\in \{1,\cdots ,t\}$ denote by $m_i$ the sum of all multiplicities
of $\mathcal B$ in the $i$-th column, i.e.
$$
m_i=\sum_{is\in\mathcal B} m_{is, j+1},\quad m_1+\dots+m_t=b.
$$
In this step we construct a relation with the leading term
\begin{equation*}
X_{1t}(-j-1)\sp{m_1}X_{2t}(-j-1)\sp{m_2}\dots X_{tt}(-j-1)\sp{m_t}X_{tt}(-j)\sp{a}.
\end{equation*}
We consider two cases: $t\in \{1,\cdots , n\}$ and $t\in \{\underline{n}, \cdots , \underline{1}\}$
\bigskip

\noindent {\bf Case 1.} \ Let $t=r$, $r\in \{1,\cdots , n\}$. Set
\begin{equation}\label{E:12.upper barrier 1}
Z_1=[r \underline{1}]\sp{2a+m_1+m_r} [r-1, \underline{1}]\sp{m_{r-1}}\dots  [3 \underline{1}]\sp{m_3} [2 \underline{1}]\sp{m_2}\,Z_0.
\end{equation}
(Note that all listed arrows mutually commute.) Then for some $c\neq0$
\begin{equation}\label{E:12.upper barrier 2}
Z_1=c\,X_{1r}(-j-1)\sp{m_1+m_r}X_{1,r-1}(-j-1)\sp{m_{r-1}}\dots X_{12}(-j-1)\sp{m_2}X_{rr}(-j)\sp{a}+\sum_{M\in\mathcal P}c_MM.
\end{equation}
The first term we obtain by applying all arrows listed in (\ref{E:12.upper barrier 1}) on factors $X_{11}(-j)$, $X_{1r}(-j)$ and $X_{11}(-j-1)$
in the following way (see Lemma \ref{L: arrows 1}):  $2a$ arrows $[r \underline{1}]$ move $X_{11}(-j)\sp a$ to $X_{rr}(-j)\sp a$, i.e. for some $c'\neq0$
\begin{equation}\label{E:12.upper barrier 3}
[r \underline{1}]\sp{2a}X_{11}(-j)\sp{a}=c'\,X_{rr}(-j)\sp{a},
\end{equation}
then $m_2$ arrows $[2 \underline{1}]$ move $m_2$ factors $X_{11}(-j-1)$ to $X_{12}(-j-1)\sp{m_2}$, and so on until
$m_1+m_r$ arrows $[r \underline{1}]$ move $m_1+m_r$ factors $X_{11}(-j-1)$ to $X_{1r}(-j-1)\sp{m_1+m_r}$.
Multiple choice of factors on which any of arrows may act give a proportionality constant $c\neq 0$, and any other way that arrows act
give in (\ref{E:12.upper barrier 2})  a combination of other monomials $M$.

It will be convenient to say that an arrow
$[\beta]=\text{ad}\, X_\beta$ is {\it misused} if it is not used to produce the first term in (\ref{E:12.upper barrier 2})  in the way described above.

If we act with less than $2a$ arrows $[r \underline{1}]$ on $X_{11}(-j)\sp{a}$, we get (cf. Lemmas \ref{L: arrows 1} and \ref{L: arrows 2}) a
monomial $M'$ with a factor $X_\beta(-j)$, $\beta\in\triangle_r\setminus \{rr\}$, and hence $M'\succ X_{11}(-j)\sp{a}$. So assume that we have
used $2a$ arrows $[r \underline{1}]$ on $X_{11}(-j)\sp{a}$, as in (\ref{E:12.upper barrier 3}), but that we misused the rest of arrows by acting
on some other factor different from $X_{11}(-j-1)$. Then we get a monomial $M$ with a factor $X_{11}(-j-1)$, and this implies that $M$ is
greater then the first term in
(\ref{E:12.upper barrier 2}), that is
\begin{equation}\label{E:12.upper barrier 4}
\ell\!\it t\,  Z_1=X_{1r}(-j-1)\sp{m_1+m_r}X_{1,r-1}(-j-1)\sp{m_{r-1}}\dots X_{12}(-j-1)\sp{m_2}X_{rr}(-j)\sp{a}.
\end{equation}
Also note (cf. Lemmas \ref{L: arrows 1} and \ref{L: arrows 2}) that all such monomials $M$ have all other factors $X_{\beta}(-j-1)$
with $\beta$ in the triangle $\triangle_r$, and at least
one factor $X_{lm}(-j-1)$  with $l,m\in\{2,\dots,r-1\}$.

Now set
\begin{equation}\label{E:12.upper barrier 5}
Z_2=[r \underline{1}]\sp{m_2+m_3+\dots+m_r} \,Z_1.
\end{equation}
Then for some $c\neq0$
\begin{equation}\label{E:12.upper barrier 6}
Z_2=c\,X_{1r}(-j-1)\sp{m_1}X_{2r}(-j-1)\sp{m_{2}}\dots X_{rr}(-j-1)\sp{m_{r}}X_{rr}(-j)\sp{a}+\sum_{M\in\mathcal P}c_MM.
\end{equation}
The first term we obtain by applying all arrows listed in (\ref{E:12.upper barrier 5}) on the first term in (\ref{E:12.upper barrier 2}) in
the following way:  first $m_2$ arrows $[r \underline{1}]$ move $X_{12}(-j-1)\sp{m_{2}}$ to $X_{2r}(-j-1)\sp{m_2}$, and so on until
$m_r$ arrows $[r \underline{1}]$ move $X_{1r}(-j-1)\sp{m_{r}}$ to $X_{rr}(-j-1)\sp{m_r}$.

Note that arrows $[r \underline{1}]$
cannot move $X_{rr}(-j)\sp a$ and that any arrow spent on $X_{1s}(-j)$, $s\in\{2,\dots,r-1\}$, will produce a monomial $M'$ greater
than $X_{rr}(-j)\sp{a}$.

What is left to consider is the case when $2a$ arrows $[r \underline{1}]$ are used as in (\ref{E:12.upper barrier 3}), and the rest of
$m_2+\dots+m_r$ arrows are not used as above. Then the action of $m_2+\dots+m_r$ arrows $[r \underline{1}]$ on the first
term in (\ref{E:12.upper barrier 2}) will produce a monomial $M$ with a factor  $X_{1s}(-j-1)$, $s\in\{2,\dots,r\}$, or a factor
$X_{1r}(-j-1)\sp p$, $p>m_1$, in either case a monomial greater than the first term in (\ref{E:12.upper barrier 6}).

On the other hand, we consider three cases of the action of $m_2+\dots+m_r$ arrows $[r \underline{1}]$ on some
$M$ in the second summand (\ref{E:12.upper barrier 2}): (1) arrows will not move a factor $X_{ps}(-j-1)$ for some
$p,s\in\{2,\dots,r-1\}$ in the case when $M$ in (\ref{E:12.upper barrier 2}) is obtained by misusing any arrow of the
form $[s\underline{1}]$, $s\in\{2,\dots,r-1\}$, (2) if some of the arrows $[r \underline{1}]$ is misused, $M$ will
have a factor  $X_{1s}(-j-1)$ for some $s\in\{2,\dots,r-1\}$,  or (3) $M$ will have a factor $X_{1r}(-j-1)\sp{m_1+p}$
with $p\geq 1$, we may briefly say that $M$ has {\it an extra} factor $X_{1r}(-j-1)$. In either of these cases $M$ is
greater than the first term (\ref{E:12.upper barrier 6}), i.e.
\begin{equation*}
\ell\!\it t\,  Z_2=X_{1r}(-j-1)\sp{m_1}X_{2r}(-j-1)\sp{m_{2}}\dots X_{rr}(-j-1)\sp{m_{r}}X_{rr}(-j)\sp{a}.
\end{equation*}
\bigskip

\noindent {\bf Case 2.} \  Let $t=\underline{r}$, $r\in \{1,\cdots , n\}$. As in (\ref{E:12.upper barrier 1}) we set
\begin{equation*}
Z_1=[\underline{r}\underline{1}]\sp{2a+m_1+m_{\underline{r}}}\prod_{s=2}\sp{\underline{r+1}}
[s\underline{1}]\sp{m_{s}}\,Z_0.
\end{equation*}
(Note that some listed arrows do not commute, for example $[s\underline{1}]$ and  $[\underline{s}\underline{1}]$.) Then for some $c\neq0$
\begin{equation}\label{E:12.upper barrier 12}
Z_1=c\,X_{1\underline{r}}(-j-1)\sp{m_1+m_{\underline{r}}}\prod_{s=2}\sp{\underline{r+1}}X_{1s}(-j-1)\sp{m_{s}}X_{\underline{r}\underline{r}}(-j)\sp{a}+\sum_{M\in\mathcal P}c_MM.
\end{equation}
By arguing in the same way as in Case 1, we see that the first term in  (\ref{E:12.upper barrier 12}) is the leading term $\ell\!\it t\,  Z_1$ of $Z_1$.

From $Z_1$ we want to construct a relation with the leading term
\begin{equation}
\label{E:12.upper barrier 13}
\Pi=\prod_{s=1}\sp{\underline{r}}X_{s\underline{r}}(-j-1)\sp{m_{s}}X_{\underline{r}\underline{r}}(-j)\sp{a}.
\end{equation}
We'll proceed as in the Case 1, by using arrows $[\underline{r} \underline{1}]$, except for the factor $X_{r\underline{r}}(-j)\sp{m_r}$
where we encounter the following difficulty: the action
$$
[\underline{r} \underline{1}] \,X_{1r}=[X_{-\varepsilon_1-\varepsilon_r}, X_{\varepsilon_1+\varepsilon_r}]
=-(\varepsilon_1+\varepsilon_r)\sp\vee=-\sum_{1\leq s<r}\alpha_s\sp\vee-2\sum_{r\leq s\leq n}\alpha_s\sp\vee,
$$
or written in our notation,
$$
[\underline{r} \underline{1}] \,X_{1r}=-\sum_{1\leq s<r} X_{s\underline{s}}-2\sum_{r\leq s\leq n}X_{s\underline{s}}
$$
(see Lemma \ref{L: arrows 3}), produce terms $X_{1\underline{1}}\prec \dots\prec X_{r-1\underline{r-1}}\prec X_{r\underline{r}}$.
So the action of arrow
$[\underline{r} \underline{1}]$ on $X_{1r}(-j-1)$ produce factors
$X_{s\underline{s}}(-j-1)$, $s\in\{1,\dots,r-1\}$, which in turn may produce monomials $M$ smaller than $\Pi$. In order to avoid this
difficulty, we use the actions
$$
[r \underline{1}] \,X_{1r}=c\, X_{rr},\qquad
[\underline{r} \underline{r}] \,X_{rr}=-(2\varepsilon_r)\sp\vee=-X_{r\underline{r}}-\dots-X_{n\underline{n}},
$$
(for some $c\neq 0$; see again Lemma \ref{L: arrows 3}). So the action of arrow $[\underline{r} \underline{r}]$ on $X_{rr}(-j-1)$ produce factors
$X_{s\underline{s}}(-j-1)$, $s\in\{r,\dots,n\}$, which are in the $r$-th row in the basis $B$ or above it.

For this reason we "first act" on a factor $X_{1r}(-j-1)\sp{m_r}$ to "empty the place" $1r$ in the first column of the basis $B$:
$$
[r\underline{1}]\sp{m_r} \,X_{1r}(-j-1)\sp{m_r}=c\, X_{rr}(-j-1)\sp{m_r}
$$
(for some $c\neq 0$), and then move further with arrows
$$
[\underline{r}\underline{r}]\sp{m_r} \,X_{rr}(-j-1)\sp{m_r}=c'\, X_{r\underline{r}}(-j-1)\sp{m_r}+\sum_{M\in\mathcal P}c_MM.
$$
(for some $c'\neq 0$). After that, as in Step 1, we shall move factors $X_{1s}(-j-1)\sp{m_s}$, $s\neq r$, from the first column to
factors $X_{s\underline{r}}(-j-1)\sp{m_s}$ in the $\underline{r}$-th row by using arrows $[\underline{r} \underline{1}]$ (see Lemma \ref{L: arrows 4}). This
reasoning motivates us to set
\begin{equation*}
Z_2=[\underline{r} \underline{r}]\sp{m_r}  [r \underline{1}]\sp{m_r}
\prod_{s=2, s\neq r}\sp{\underline{r}}[\underline{r} \underline{1}]\sp{m_{s}}\,Z_1.
\end{equation*}
Then
\begin{equation*}
Z_2=c\,\Pi+\sum_{M\in\mathcal P}c_MM
\end{equation*}
for some $c\neq0$ and $\Pi=\ell\!\it t\,  Z_2$, where  $\Pi$ is given by (\ref{E:12.upper barrier 13}). By arguing as in the Case 1, we obtain $\Pi$ by acting with arrows
on factors in the following way: for $ s\neq r$
$$
[\underline{r} \underline{1}]\sp{m_{s}}\,X_{1s}(-j-1)\sp{m_s}=c_s''\, X_{s\underline{r}}(-j-1)\sp{m_s}+\dots
$$
for some $c_s''\neq 0$, and
$$
[\underline{r} \underline{r}]\sp{m_r}[r \underline{1}] \sp{m_r}\,X_{1r}(-j-1)\sp{m_r}=c'c\, X_{r\underline{r}}(-j-1)\sp{m_r}+\dots.
$$
We'll say that any other way of using arrows on $\ell\!\it t\,  Z_1$ {\it is misusing arrows}.

If arrows $[r\underline{1}]$ are misused, we produce a factor $X_{sr}(-j-1)$ for $s\in\{2,\dots,r-1\}$, or
$X_{rs}(-j-1)$  for $s\in\{r+1,\dots, \underline{r+1}\}$.

If arrows $[\underline{r} \underline{r}]$ are misused, we produce a factor $X_{rr}(-j-1)$ or an extra factor $X_{r\underline{r}}(-j-1)$

If arrows $[\underline{r}\underline{1}]$ are misused, we produce a factor $X_{rs}(-j-1)$ for some $s\neq 1,r,\underline{r} $
or an extra factor $X_{1\underline{r}}(-j-1)$. In either of these cases we obtan a monomial $M$ greater than $\Pi$.

The analysis of how the arrows  $[\underline{r} \underline{1}]$, $[\underline{r} \underline{r}]$, or $[r \underline{1}]$
act on monomials $M$ in the second summand in (\ref{E:12.upper barrier 12}) is analogous to the Case 1.
\bigskip

\noindent {\bf Construction of upper cascade.} \, From the upper barrier
\begin{equation}\label{E:12 upper waterfalls 1}
Z_2=Z_{2;t}=\prod_{s=1}\sp{t}X_{st}(-j-1)\sp{m_{s}}X_{tt}(-j)\sp{a}+\sum_{M\in\mathcal P}c_MM
\end{equation}
in a sequence of steps we construct a relation $Z_{2;q}$ with the leading term
\begin{equation}\label{E:12 upper waterfalls 2}
\prod_{\beta \in \mathcal{B}} X_{\beta}(-j-1)^{m_{\beta,j+1}}\ X_{qq}(-j)^{a},
\end{equation}
where $q$ is the index of the lowest row in which the upper cascade $\mathcal B$ has a point, and $t$ is the index of the highest
row in which $\mathcal B$ has a point. As above, for $i\in \{1,\cdots ,t\}$ we denote by $m_i$ the sum of all multiplicities of
$\mathcal B$ in the $i$-th column, i.e.
$$
m_i=\sum_{is\in\mathcal B} m_{is, j+1},\quad m_1+\dots+m_t=b.
$$
Let $p\in \{1,\cdots ,t\}$ be the index of the first column from the left for which the multiplicity in $\mathcal B$ is not zero,
$$
m_{st;j+1}=0\quad \text{for}\quad s\succ p,\qquad 0<m_{pt;j+1}\leq m_p,\qquad m_{st;j+1}=m_s\quad \text{for}\quad p\succ s .
$$
Let $t'$ be the index of the first row in the basis $B$ below $t$-th row. Then we act with arrows $[t'\underline{t}]$ corresponding to  a
simple root (cf. Lemmas \ref{L: arrows 7} and \ref{L: arrows 8}), moving factors from
$t$-th row to the next row bellow "like in waterfalls"
\begin{equation}\label{E:12 upper waterfalls 3}
\begin{aligned}
&Z_{2;t'}=[t'\underline{t}]\sp{2a+m_1+\dots+m_p-m_{pt;j+1}}Z_{2;t}\\
=c\prod_{s\succ p}&X_{st'}(-j-1)\sp{m_{s}}X_{pt'}(-j-1)\sp{m_{p}-m_{pt;j+1}}\prod_{p\succeq s}X_{st}(-j-1)\sp{m_{st;j+1}}X_{t't'}(-j)\sp{a}\\
&+\sum_{M\in\mathcal P}c_MM.
\end{aligned}
\end{equation}
It is easy to see that the first summand in (\ref{E:12 upper waterfalls 3})  is obtained by using arrows $[t'\underline{t}]$ on the
first summand in (\ref{E:12 upper waterfalls 1}), moving first $X_{tt}(-j)\sp{a}$ to $X_{t't'}(-j)\sp{a}$, then "one by one" factors
$X_{st}(-j-1)\sp{m_{s}}$ to $X_{st'}(-j-1)\sp{m_{s}}$, and finally, $m_{p}-m_{pt;j+1}$ factors $X_{pt}(-j-1)$ to $X_{pt'}(-j-1)$.
It is also easy to see that (apart from the constant $c\neq 0$) the first summand in (\ref{E:12 upper waterfalls 3}) is the leading term of
$Z_{2;t'}$. We say that we constructed {\it a waterfall at $p$-th column}.

Note that arrows we used correspond to simple roots, i.e., $[t'\underline{t}]=\text{ad}\,X_{-\alpha_{r}}$ if $t=r$ and $[t'\underline{t}]
=\text{ad}\,X_{-\alpha_{r-1}}$ if $t=\underline{r}$ for $r\in \{1,\dots ,n\}$. If
$$
m_s=0\quad\text{for}\quad s\succ p,
$$
then we move factors row by row by using arrows corresponding to simple roots, producing a relation $Z_{2;q}$ with the leading term
(\ref{E:12 upper waterfalls 2}):
$$
\prod_{r=t}\sp qX_{pr}(-j-1)\sp{m_{pr;j+1}}\prod_{p\succ s}X_{st}(-j-1)\sp{m_{st;j+1}}X_{qq}(-j)\sp{a}.
$$
If
$$
m_s\neq 0\quad\text{for some}\quad s\succ p,
$$
We take the smallest such $p''\succ p$ and we find the first row below $t$-th row, say $t''$-th row, such that
$$
m_{p''t'';j+1}\neq 0.
$$
Then we move, step by step, all the factors in $i$-th column, $i\prec p$, from $t'$-th row to $t''$-th row,  and create a waterfall
at $p$-th column:
$$
\prod_{s\succ p}X_{st''}(-j-1)\sp{m_{s}}\prod_{r=t}\sp{t''}X_{pr}(-j-1)\sp{m_{pr;j+1}}
\prod_{p\succ s}X_{st}(-j-1)\sp{m_{st;j+1}}X_{qq}(-j)\sp{a}.
$$
In such a way we proceed and construct a relation with the leading term   (\ref{E:12 upper waterfalls 2}).
\bigskip

\noindent{\bf Preparation of lower barrier.} \  Let $t\in \{q \dots , \underline{1}\}$ be the index of uppermost row in which the lower cascade
$\mathcal A$ in (\ref{E:12.12}) has
a point. For $i\in \{q,\cdots ,t\}$ denote by $m_i$ the sum of all multiplicities of $\mathcal A$ in the $i$-th column, i.e.
$$
m_i=\sum_{is\in\mathcal A} m_{is, j},\quad m_q+\dots+m_t=a.
$$
In this step we construct  a lower barrier, i.e. a relation with the leading term
\begin{equation}\label{E:12 lower barrier 1}
\prod_{\beta \in \mathcal{B}} X_{\beta}(-j-1)^{m_{\beta,j+1}}\,X_\text{qt}(-j)\sp{m_q}\dots X_{tt}(-j)\sp{m_t}.
\end{equation}

In previous steps we have acted with arrows on leading terms of relations in such a way that
"the first" $2a$ arrows were "spent" on monomials in degree $-j$, i.e.
$$
[s\underline{r}\,]\sp{2a}X_{rr}(-j)\sp{a}=c'X_{ss}(-j)\sp{a},
$$
and any further action of arrows did not move monomial $X_{ss}(-j)\sp{a}$. Hence with the rest of arrows we could construct
the upper cascades (\ref{E:12 upper waterfalls 2}).

In this step we'll obtain the leading term (\ref{E:12 lower barrier 1}) by acting with arrows only on factors  $X_\gamma(-j-1)$ with
$\gamma$ in a lower triangle
${}\sp q\!\triangle$ , while the action on factors  $X_\beta(-j-1)$ will produce greater terms. For this reason here we'll omit writing
factors in degree $-j-1$.

Let $q'\prec q$ be the index of the first column (row) next to the $q$-th column (row). As before we see that
$$
X_{qt}(-j)\sp{m_q+m_t}\dots X_{qq'}(-j)\sp{m_{q'}}=\ell\!\it t\Big([t\underline q]\sp{m_q+m_t}\dots[q'\underline q]\sp{m_{q'}}X_{qq}(-j)\sp{a}\Big).
$$
If the set of points $\{qt,q't,\dots,tt\}$ in the lower triangle ${}\sp q\!\triangle$ does not contain the point $\underline{t}t$, then we apply
$$
[t\underline q]\sp{m_{q'}+\dots+m_t}
$$
to produce a relation with the leading term (\ref{E:12 lower barrier 1}). On the other hand, if the set of points $\{qt,q't,\dots,tt\}$ contains the point
$\underline{t}t$, then we should modify our construction---as in the Case 2 above---and, for moving the factor
$X_{\underline{t}q}(-j)\sp{m_{\underline{t}}}$
to  $X_{\underline{t}t}(-j)\sp{m_{\underline{t}}}$  via $X_{\underline{t}\underline{t}}(-j)\sp{m_{\underline{t}}}$, we act with arrows
$$
[tt]\sp{m_{\underline t}}[\underline t\underline q]\sp{m_{\underline t}}
$$
instead of $[t\underline q]\sp{m_{\underline t}}$.
\bigskip

\noindent {\bf Construction of lower cascade.} \, From  a relation with the leading term (\ref{E:12 lower barrier 1}) we construct
a relation with the leading term
$$
\prod_{\beta \in \mathcal{B}} X_{\beta}(-j-1)^{m_{\beta,j+1}}\ \prod_{\alpha \in \mathcal{A}} X_{\alpha}(-j)^{m_{\alpha,j}}
$$
in a way described in the construction of upper cascade, except that here we act with arrows only on factors  $X_\gamma(-j-1)$ with
$\gamma$ in a lower triangle
${}\sp q\!\triangle$ , while the action on factors  $X_\beta(-j-1)$ will produce greater terms.
%
\end{proof}

\section{Conjectured colored Rogers-Ramanujan type identities}

By using Theorem  \ref{T: leading terms} we can explicitly describe the reduced spanning set in Proposition \ref{P: spanning set of standard module}. We guess that this spanning set is a basis:
\bigskip

\noindent {\bf Conjecture 1.} {\it Let $n\geq 2$ and $k\geq 2$. We consider the standard module $L(k\Lambda_0)$ for the affine Lie
algebra of type $C_n\sp{(1)}$ with the basis
$$
\{X_{ab}(j)\mid  {ab}\in B,\   j\in\mathbb Z \}
\cup\{c, d\},
$$
where $B=\{ab \mid    b\in\{1,2,\cdots ,n,\underline{n},\cdots ,\underline{2},\underline{1}\},\ a\in\{1,\cdots ,b\} \}$.

We conjecture that the set of monomial vectors
\begin{equation}\label{E: conjecture 1a}
\prod_{ab\in B, \,j>0}X_{ab}(-j)\sp{m_{ab;j}}\,v_0,
\end{equation}
satisfying difference conditions
\begin{equation*}
\sum_{ab\in\mathcal B} m_{ab;j+1}+\sum_{ab\in\mathcal A} m_{ab;j}\leq k
\end{equation*}
for any admissible pair of cascades $(\mathcal B,\mathcal A)$, is a basis of $L(k\Lambda_0)$.}
\bigskip

As already mentioned, the conjecture is true for $n=1$ and all $k\geq 1$ \cite{MP3} and for $k=1$ for all $n\geq 2$ \cite{P\v S}. We have also checked by hand the corresponding combinatorial identity below for partitions of $m=1,\dots, 8$ in the case $n=k=2$.

If our conjecture is true, then we have a combinatorial Rogers-Ramanujan type identities by using Lepowsky's product
formula for principaly specialized characters of standard modules (see \cite{L}, cf. \cite{M2}, \cite{MP3}). In the case of
$n=2$ and $k\geq 1$ we have product formulas for principally specialized characters of standard $C_2\sp{(1)}$-modules
$L(k\Lambda_0)$
\begin{equation}\label{E: conjecture 2}
\prod_{\substack{j\geq 1\\j\not\equiv 0\,\textrm{mod}\,2}}\!\!\frac{1}{1-q^j}
\prod_{\substack{j\geq 1\\j\not\equiv 0,\pm1,\pm2,\pm3\,\textrm{mod}\,2k+6}}\!\!\frac{1}{1-q^j}
\prod_{\substack{j\geq 1\\j\not\equiv 0,\pm1,\pm(k+1),\pm(k+2),k+3\,\textrm{mod}\,2k+6}}\!\!\frac{1}{1-q^j}.
\end{equation}
This product can be interpreted combinatorially in the following way: For fixed $k$ let $\mathcal C_k$ be a disjoint union of
integers in three colors, say $j_1, j_2, j_3$ is the integer $j$ in colors $1,2,3$, satisfying the following congruence conditions
\begin{equation}\label{E: conjecture 3}
\begin{aligned}
&\{j_1\mid j\geq 1, j\not\equiv 0\,\textrm{mod}\,2\}, \\
&\{j_2\mid j\geq 1, j\not\equiv 0,\pm1,\pm2,\pm3\,\textrm{mod}\,2k+6\}, \\
&\{j_3\mid j\geq 1, j\not\equiv 0,\pm1,\pm(k+1),\pm(k+2),k+3\,\textrm{mod}\,2k+6\}. \\
\end{aligned}
\end{equation}
Set $|j_a|=j$. If we expand the product (\ref{E: conjecture 2}) in Taylor series, then the  coefficient of $q^m$ can be interpreted
as a number of colored partitions of $m$
\begin{equation}\label{E: conjecture 4}
m=\sum_{j_a\in\mathcal C_k}j_af_{j_a}.
\end{equation}
(To be correct, in (\ref{E: conjecture 4}) we should write $m=\sum|j_a|f_{j_a}$.)
For example, for $k=2$ we have
\begin{equation*}
\mathcal C_2=\{1_1,  3_1, 5_1, 7_1,\dots\}\sqcup\{4_2, 5_2, 6_2,14_2,\dots\}\sqcup\{2_3, 8_3 , 12_3, 18_3\dots\};
\end{equation*}
all ordinary partitions of $m=5$ are
$$
5, \ 4+1, \ 3+2, \ 3+1+1, \ 2+2+1, \ 2+1+1+1, \ 1+1+1+1+1,
$$
and all colored partitions of $5$ with colored parts in $\mathcal C_2$ are
$$
5_1,\ 5_2, \
 \ 4_2+1_1,
 \ 3_1+2_3,
 \ 3_1+1_1+1_1, \ 2_3+2_3+1_1, \ 2_3+1_1+1_1+1_1, \ 1_1+1_1+1_1+1_1+1_1.
$$

On the other hand, in the principal specialization $e\sp{-\alpha_i}\mapsto q\sp1$, $i=0,1,2$, the sequence of root subspaces in $C_2\sp{(1)}$
\begin{equation}\label{E: conjecture 5}
X_{ab}(-1), ab\in B,\quad X_{ab}(-2), ab\in B,\quad X_{ab}(-3), ab\in B,\quad \dots
\end{equation}
obtains degrees
\begin{equation}\label{E: conjecture 6}
\begin{array}{cccc}
1 &  &&   \\
2 & 3 & &  \\
3 & 4 & 5 & \\
4 & 5 & 6 & 7  \\
\end{array}
\begin{array}{cccc}
5 &  &&   \\
6 & 7 & &  \\
7 & 8 & 9 & \\
8 & 9 & 10 & 11  \\
\end{array}
\begin{array}{cccc}
9  &  &&   \\
10 & 11& &  \\
11 &  12& 13 & \\
12& 13 & 14& 15  \\
\end{array}
\dots
\end{equation}
One way or the other, in (\ref{E: conjecture 6})  we see almost two sequences of natural numbers and almost one sequence of odd numbers.
In order to make numbers distinct, we consider four colors $1,2,3,4$, say
\begin{equation}\label{E: conjecture 7}
\begin{array}{cccc}
1_1 &  &&   \\
2_2 & 3_2 & &  \\
3_3 & 4_3 & 5_3 & \\
4_4 & 5_4 & 6_4 & 7_4  \\
\end{array}
\begin{array}{cccc}
5_1 &  &&   \\
6 _2& 7_2 & &  \\
7_3 & 8_3 & 9_3 & \\
8_4 & 9_4 & 10_4 & 11_4  \\
\end{array}
\begin{array}{cccc}
9_1  &  &&   \\
10 _2 & 11 _2& &  \\
11_3 &  12_3& 13_3 & \\
12_4& 13_4 & 14_4& 15_4  \\
\end{array}
\dots,
\end{equation}
so that numbers in the first row have color 1, numbers in the second row have color 2, and so on. In other words, for fixed $n=2$ we
consider a disjoint  union $\mathcal D_2$  of integers in four colors, say $j_1,  j_2,  j_3,  j_4$ is the integer $j$ in colors $1,2,3,4$,
satisfying the congruence conditions
\begin{equation}\label{E: conjecture 8}
\begin{aligned}
&\{j_1\mid j\geq 1, j\equiv 1\,\textrm{mod}\,4\}, \\
&\{j_2\mid j\geq 2, j\equiv 2,3\,\textrm{mod}\,4\}, \\
&\{j_3\mid j\geq 3, j\equiv 0,1,3\,\textrm{mod}\,4\}, \\
&\{j_4\mid j\geq 4, j\equiv 0,1,2,3\,\textrm{mod}\,4\} \\
\end{aligned}
\end{equation}
and arranged in a sequence of triangles (\ref{E: conjecture 7}).
For adjacent triangles in  (\ref{E: conjecture 7}) corresponding to
$$
\dots,\quad X_{ab}(-j), ab\in B,\quad X_{ab}(-j-1), ab\in B,\quad\dots
$$
in  (\ref{E: conjecture 5}) and a fixed row $r$ we consider the corresponding two triangles: ${}\sp r\!\triangle$ on the left and $\triangle_r$
on the right. For example, for the third row we have $r=\underline{2}$ and two triangles denoted by bullets
\begin{equation}\label{E: conjecture 9}
\dots
\begin{array}{cccc}
\cdot &  &&   \\
\cdot & \cdot & & \\
\cdot &  \cdot&  \cdot&  \\
 \cdot &  \cdot& \cdot & \cdot  \\
\end{array}
\begin{array}{cccc}
\cdot &  &&   \\
\cdot & \cdot & & \\
\cdot &  \cdot&  \bullet  &  \\
 \cdot &  \cdot& \bullet  & \bullet   \\
\end{array}
\begin{array}{cccc}
\bullet  &  &&   \\
\bullet  & \bullet  & & \\
\bullet  &  \bullet  &  \bullet  &  \\
 \cdot &  \cdot& \cdot & \cdot  \\
\end{array}
\begin{array}{cccc}
\cdot &  &&   \\
\cdot & \cdot & & \\
\cdot &  \cdot&  \cdot&  \\
 \cdot &  \cdot& \cdot & \cdot  \\
\end{array}
\dots
\end{equation}
are \ ${}\sp{\underline{2}}\triangle$ on the left and $\triangle_{\underline{2}}$ on
the right. We say that two cascades
$$
\mathcal A\subset {}\sp r\!\triangle\quad\text{and}\quad\mathcal B\subset \triangle\sb r
$$
form an admissible pair of cascades in the sequence  (\ref{E: conjecture 7}).

If our Conjecture 1 is correct, then  the  coefficient of $q^m$ in the principally specialized character of $L(k\Lambda_0)$
equals the number of basis vectors (\ref{E: conjecture 1a}) of degree $-m$, i.e., equals the number of colored partitions of $m$
\begin{equation}\label{E: conjecture 10}
m=\sum_{j_a\in\mathcal D_2}j_af_{j_a}
\end{equation}
satisfying difference conditions
\begin{equation}\label{E: conjecture 11}
\sum_{j_a\in\mathcal A} f_{j_a}+\sum_{j_b\in\mathcal B} f_{j_b}\leq k:
\end{equation}
for every admissible pair of cascades in the sequence  (\ref{E: conjecture 7}).
\bigskip

\begin{example}
Let $n=k=2$. Then the first nine terms of Taylor series (\ref{E: conjecture 2}) are
\begin{equation}\label{Taylor}
 1+q+2q^2+3q^3 + 5q^4 +8q^5 + 12q^6+17q^7+25q^8+\cdots.
\end{equation}
By enumerating all  admissible cascades for the basis $B$ of $C_2$ we made a list of $4\times 8=32$ difference conditions. From the list of difference conditions and the list of ordinary partitions,  direct calculation gives all colored partitions of  $m=1, 2, \cdots, 8$ with colored parts in $\mathcal D_2$:
\begin{eqnarray}\label{colored patrition}
1 & = & 1_1\nonumber\\
2 & = & 2_2 = 1_1+1_1\nonumber\\
3 & = & 3_2 = 3_3 =  2_2 + 1_1\nonumber\\
4 & = & 4_3 = 4_4 = 3_2 + 1_1 = 3_3 + 1_1 = 2_2 + 2_2 \nonumber\\
5 & = & 5_1 = 5_3 = 5_4 = 4_3 + 1_1 = 4_4 + 1_1 = 3_2 + 2_1 = 3_3 + 2_1 = 3_2 + 1_1 + 1_1\nonumber\\
6 & = & 6_2 = 6_4 = 5_1 + 1_1 = 5_3 + 1_1 = 5_4 + 1_1 = 4_3 + 2_2 = 4_4 + 2_2 = 4_3 + 1_1 + 1_1\nonumber\\
7 & = & 7_2 = 7_3 = 7_4 = 6_2 + 1_1 = 6_4 + 1_1 = 5_1 + 2_2 = 5_3 + 2_2 = 5_4 + 2_2 = 5_3 + 1_1 + 1_1\nonumber\\
  & = & 5_4 + 1_1 + 1_1 = 4_3 + 3_2 = 4_3 + 3_3 = 4_4 + 3_2 = 4_4 + 3_3 = 4_3 + 2_2 +1_1 \nonumber\\
  & = & 3_2 + 3_2 +1_1 = 3_2 + 3_3 + 1_1 \nonumber\\
8 & = & 8_3 = 8_4 = 7_2 + 1_1 = 7_3 + 1_1 = 7_4 +1_1 = 6_2 + 2_2 = 6_4 + 2_2 = 6_2 + 1_1 + 1_1 \nonumber\\
  & = & 6_4 + 1_1 + 1_1 =   5_1 + 3_2 = 5_1 + 3_3 = 5_2 + 3_2 = 5_2 + 3_3= 5_3 + 3_2 = 5_3 + 3_3 \nonumber\\
  & = & 5_3 + 2_2 + 1_1 = 5_4 + 2_2 + 1_1 = 4_3 + 4_3 = 4_3 + 4_4 = 4_4 + 4_4 = 4_3+3_2+1_1\nonumber\\
  & = & 4_3+3_3+1_1 = 4_4+3_2+1_1 = 4_3+2_2+2_2 = 3_2+3_2 +1_1+1_1\nonumber\ .
\end{eqnarray}
Hence the number of partitions (\ref{E: conjecture 10}) satisfying difference conditions  (\ref{E: conjecture 11}) coincides with the coefficients of above Taylor series
(\ref{Taylor}) for $m=1,2,\cdots,8$.

We omit the details of  calculations above; we'll only explain how difference conditions eliminated the colored partition $5_1 + 2_2 + 1_1$  in the case $m= 8$. First of all, notice that   $5_1$ belongs to the triangle $X_{ab}(-2)$, and $2_2$ and $ 1_1$   belong to the triangle
$X_{ab}(-1)$ (see \ref{E: conjecture 7}). Now we chose $r=1$ and consider the triangles \ ${}\sp{1}\triangle$ and $\triangle_{1}$

\begin{equation*}
\begin{array}{cccc}
\bullet  &  &&   \\
\bullet  & \bullet  & & \\
\bullet  &  \bullet  &  \bullet  &  \\
\bullet  &  \bullet  &  \bullet  & \bullet \\
\end{array}
\begin{array}{cccc}
\bullet &  &&   \\
\cdot & \cdot & & \\
\cdot &  \cdot&  \cdot&  \\
 \cdot &  \cdot& \cdot & \cdot  \\
\end{array}
\end{equation*}
One pair of admissible cascades is
\begin{equation*}
\begin{array}{cccc}
m_{11;1}  &  &&   \\
m_{12;1}  & \cdot  & & \\
m_{1\underline{2};1}  &  \cdot  &  \cdot  &  \\
m_{1\underline{1};1}  &  \cdot  &  \cdot  & \cdot\\
\end{array}
\begin{array}{cccc}
m_{11;2} &  &&   \\
\cdot & \cdot & & \\
\cdot &  \cdot&  \cdot&  \\
 \cdot &  \cdot& \cdot & \cdot  \\
\end{array}
\end{equation*}
and the  corresponded difference condition---one of 32 conditions---is given by
$$m_{11;2} +m_{11;1} + m_{12;1} + m_{1\underline{2};1} + m_{1\underline{1};1}\leq 2\ . $$
Since $m_{11;2} +m_{11;1} + m_{12;1} + m_{1\underline{2};1} + m_{1\underline{1};1}= 1 + 1 + 1+0+0 =3>2$,
the observed colored partition is eliminated from the list.
\end{example}

\noindent {\bf Conjecture 2.} {\it Let $n=2$ and $k\geq 2$. We conjecture that for every $m\in\mathbb N$ the number of colored partitions
$$
m=\sum_{j_a\in\mathcal C_k}j_af_{j_a}
$$
in three colors satisfying congruence conditions (\ref{E: conjecture 3}) equals  the number of colored partitions
$$
m=\sum_{j_a\in\mathcal D_2}j_af_{j_a}
$$
in four colors satisfying congruence conditions (\ref{E: conjecture 8}) and difference conditions (\ref{E: conjecture 11}) for every admissible pair of cascades in the sequence  (\ref{E: conjecture 7}).
}

\begin{remark}
It is clear how to extend the above conjecture to $C_n\sp{(1)}$ for $n>2$. The product formulas for principally specialized characters of some $L(\Lambda_0)$  and $L(2\Lambda_0)$ are given in \cite{M2} and \cite{M4}.

In the case $n=1$ and $k\geq 1$ the product formulas for principally specialized characters of $L(k\Lambda_0)$ and the corresponding combinatorial identities are given in \cite{MP3}.

In (\ref{E: conjecture 9}) we had several choices for triangles
 \ ${}\sp{r}\!\triangle$ on the left and $\triangle_{r}$ on the right.
For $n=1$ we have only two choices: (i) \ ${}\sp{1}\triangle$ on the left and $\triangle_{1}$ on
the right, and (ii) \ ${}\sp{\underline{1}}\triangle$ on the left and $\triangle_{\underline{1}}$ on
the right:
\begin{equation}\label{E: conjecture MP}
\dots
\begin{array}{cccc}
\cdot &  \\
\cdot & \cdot  \\
\end{array}
\begin{array}{cccc}
\bullet  &    \\
\bullet  & \bullet   \\
\end{array}
\begin{array}{cccc}
\bullet &    \\
\cdot & \cdot  \\
\end{array}
\begin{array}{cccc}
\cdot &  \\
\cdot & \cdot  \\
\end{array}
\dots,
\quad
\dots
\begin{array}{cccc}
\cdot &  \\
\cdot & \cdot  \\
\end{array}
\begin{array}{cccc}
\cdot &  \\
\cdot & \bullet  \\
\end{array}
\begin{array}{cccc}
\bullet  &    \\
\bullet  & \bullet   \\
\end{array}
\begin{array}{cccc}
\cdot &    \\
\cdot & \cdot  \\
\end{array}
\dots,
\end{equation}
Moreover, for $n=1$ cascades are either vertical or horizontal. Altogether this gives four conditions (\ref{E: MP argument 6}).

In the case $n=k=1$ the corresponding identity is equivalent to one of Capparelli's identities (see \cite{A3}, \cite{C}, \cite{MP3}).

In a way analogous to  (\ref{E: conjecture 9}) and  (\ref{E: conjecture MP}), we can also visualize difference conditions
(\ref{E: LW argument 4}) in the Rogers-Ramanujan case---it is just one difference condition (\ref{E: difference condition 1}) for two adjacent points
$$
\dots\quad\cdot\quad\cdot\quad\cdot\quad\bullet\quad\bullet\quad\cdot\quad\cdot\quad\dots\,.
$$
\end{remark}



\begin{thebibliography}{FKLMM}

\bibitem[A]{A}
A. K. ~Agarwal,
{\em Rogers-Ramanujan identities for$n$-color partitions},
J. Number Theory {\bf 28} (1988), 299--305.

\bibitem[AA]{AA}
A. K. ~Agarwal and G. E. ~Andrews,
{\em Rogers-Ramanujan identities for partitions with $N$ copies
of $N$},
J. Combin. Theory Ser. A {\bf 45} (1987), 40-49.

\bibitem[A1]{A1} G. E. Andrews,
\textit{An analytic generalization of the Rogers-Ramanujan identities for odd moduli},
Proc. Nat. Acad. Sci. U.S.A. \textbf{71} (1974), 4082--4085.

\bibitem[A2]{A2}
G. E. Andrews,
{\em The theory of partitions},
Encyclopedia of Mathematics and Its Applications, Vol. 2, Addison-Wesley, 1976.

\bibitem[A3]{A3}
G. E. Andrews,
{\em Schur's theorem, Capparelli's conjecture and the $q$-trinomial coefficients},
in Proc. Rademacher Centenary Conf. (1992), Contemporary Math {\bf 167}, 1994, pp. 141--154.

\bibitem [AKS]{AKS}    E. Ardonne, R. Kedem and M. Stone,
\textit{Fermionic characters of arbitrary highest-weight integrable
$\widehat{\mathfrak sl}_{r+1}$-modules}, Comm. Math. Phys.
\textbf{264} (2006), 427--464.

\bibitem[BPT]{BPT} I. Baranovi\' c, M. Primc, G. trup\v cevi\' c, {\em Bases of Feigin-Stoyanovsky's type subspaces for  $C_\ell^{(1)}$},  accepted for publication in The Ramanujan Journal (2016 doi:10.1007/s11139-016-9840-y).  

\bibitem[BM]{BM}
A. Berkovich, B. M. McCoy,
{\em Rogers-Ramanujan identities: A century of progress from mathematics to physics},
Documenta Math, Extra volume ICM 1998 III (1998), 163--172.

\bibitem[BMS]{BMS}
A.~Berkovich, B.~McCoy and A.~Schilling, {\em Rogers-Schur-Ramanujan type
identities for the $M(p,p')$ minimal models of conformal field theory}.
Comm. Math. Phys. {\bf 191} (1998),
325--395.

\bibitem[BoM]{BoM} M.~Bos and K.C.~Misra, {\em Level two representations of
$A_7^{(2)}$ and Rogers-Ramanujan identities}, Comm. Algebra {\bf 22} (1994),
3965-3983.

\bibitem [B]{B}    N. Bourbaki,
\textit{Alg\`ebre commutative}, Hermann, Paris, 1961.

\bibitem[B1]{B1}
D. M. Bressoud,
{\em A generalization of the Rogers-Ramanujan identities for all moduli},
J. Comb. Theory Ser. A \textbf{27} (1979), 64--68.

\bibitem[B2]{B2}
D. M. Bressoud,
{\em An analytic generalization of the Rogers–Ramanujan identities with interpretation},
Quart. J. Math. Oxford  \textbf{31} (1980), 385--399.

\bibitem[CLM]{CLM}
C. Calinescu, J. Lepowsky, A. Milas,
{\em Vertex-algebraic structure of the principal subspaces of level one modules for the untwisted affine Lie algebras of types A,D,E},
J. Algebra \textbf{323} (2010), 167--192.

\bibitem [C]{C}   S. Capparelli,
\textit{On some representations of twisted affine Lie algebras and
combinatorial identities}, J. Algebra  \textbf{154} (1993),
335--355.

\bibitem [DL]{DL}    C. Dong and J. Lepowsky,
\textit{Generalized vertex algebras and relative vertex
operators}, Progress in Mathematics \textbf{112}, Birkh\"{a}user,
Boston, 1993.

\bibitem [FJLMM]{FJLMM}   B. Feigin, M. Jimbo, S. Loktev, T. Miwa and
E. Mukhin, \textit{Bosonic formulas for $(k,l)$-admissible
partitions}, math.QA/0107054; \textit{Addendum to `Bosonic formulas
for $(k,l)$-admissible partitions'}, math.QA/0112104.

\bibitem [FKLMM]{FKLMM}   B. Feigin, R. Kedem, S. Loktev, T. Miwa and E. Mukhin,
 \textit{Combinatorics of the $\widehat{\mathfrak sl}_2$ spaces of coinvariants},
Transformation Groups  \textbf{6} (2001), 25--52.

\bibitem[FS]{FS} B.~Feigin and A.~Stoyanovski\u i,
{\em Functional models of the representations of current
algebras, and semi-infinite Schubert cells},
Funct. Anal. Appl. {\bf 28} (1994), 55--72;
Quasi-particles model for the representations of Lie algebras and
geometry of flag manifolds, (hep-th/9308079).

\bibitem [F]{F}   E. Feigin, \textit{The PBW filtration},
Represent. Theory {\bf 13} (2009), 165-181.

\bibitem[FQ]{FQ} O.~Foda and Y.-H.~Quano,
{\em Polynomial identities of the Rogers-Ramanujan type},
Int. J. Mod. Phys. A {\bf 10} (1995), 2291-2315.

\bibitem  [FHL]{FHL}
 I. B. Frenkel, Y.-Z. Huang and J. Lepowsky,
\textit{On axiomatic approaches to vertex operator algebras and
modules}, Memoirs of the Amer. Math. Soc. \textbf{104}, No. 494
(1993).

\bibitem [FLM]{FLM}
I. B. Frenkel, J. Lepowsky and A. Meurman,  \textit{Vertex
operator algebras and the Monster}, Pure and Applied Math.,
Academic Press, San Diego, 1988.

\bibitem [FZ]{FZ}
I. B. Frenkel and Y. Zhu, \textit{Vertex operator algebras
associated to representations of affine and Virasoro algebras},
Duke Math. J. \textbf{66} (1992), 123--168.

\bibitem[Ge]{Ge} G.~Georgiev,
{\em Combinatorial constructions of modules for infinite-dimensional
Lie algebras, I. Principal subspace.} J. Pure Appl. Algebra {\bf 112} (1996), 247--286;
 {\em  II. Parafermionic space}, q-alg/9504024.

\bibitem[G]{G}
B. Gordon,
{\em A combinatorial generalization of the Rogers-Ramanujan identities},
Amer. J. Math. \textbf{83} (1961), 393--399.

\bibitem[JMS]{JMS}
N. Jing, K. C. Misra, C. D. Savage,
{\em On multi-color partitions and the generalized Rogers-Ramanujan identities},
Commun. Contemp. Math. Vol. {\bf 03} (2001), 533--548;

\bibitem [K]{K}
V. G. Kac, \textit{Infinite-dimensional Lie algebras} 3rd ed,
Cambridge Univ. Press, Cambridge, 1990.

\bibitem[L]{L}
J. Lepowsky,
{\em Lectures on Kac-Moody Lie algebras}, Universit\' e de Paris VI, 1978.

\bibitem[LM]{LM}
J. Lepowsky, S. Milne,
{\em Lie algebraic approaches to classical partition identities},
Adv. Math. \textbf{29} (1978), 15--59.

\bibitem[LP]{LP} J.~Lepowsky and M.~Primc,
{\em Structure of the standard modules for the affine Lie algebra $A^
{(1)}_1$}. Cont. Math., {\bf 46}. Amer. Math. Soc.,
1985.

\bibitem [LW]{LW}
J. Lepowsky and R. L. Wilson,  \textit{The structure of standard
modules, I: Universal algebras and the Rogers-Ramanujan
identities}, Invent. Math. \textbf{77} (1984),  199--290;
\textit{II: The case $A_1^{(1)}$, principal gradation}, Invent.
Math. \textbf{79} (1985),  417--442.

\bibitem [Li]{Li}
H.-S. Li, \textit{Local systems of vertex operators, vertex
superalgebras and modules}, J. of Pure and Appl. Alg. \textbf{109}
(1996), 143--195.

\bibitem[M]{M} M. ~Mandia, {\em Structure of the level one standard modules
for the affine Lie algebras $B_\ell^{(1)}$, $F_4^{(1)}$, and $G_2^{(1)}$},
Mem. Amer. Math. Soc. {\bf 65} (1987).

\bibitem[MP1]{MP1}
A. Meurman, M. Primc,
{\em Annihilating ideals of standard modules of $\mathfrak{sl}(2,\mathbb{C})\tilde{ }$ and combinatorial identities},
Adv. Math. \textbf{64} (1987), 177--240.

\bibitem [MP2]{MP2}
A. Meurman and M. Primc,  \textit{Vertex operator algebras and
representations of affine Lie algebras}, Acta Appl. Math.
\textbf{44} (1996), 207--215.

\bibitem [MP3]{MP3}
A. Meurman and M. Primc,  \textit{Annihilating fields of standard
modules of ${\mathfrak sl}(2,\mathbb C)\,\widetilde{}$ and
combinatorial identities}, Memoirs of the Amer. Math. Soc.
\textbf{137}, No. 652 (1999).

\bibitem [MP4]{MP4}
 A. Meurman and M. Primc,
\textit{A basis of the basic ${\mathfrak sl}(3,\mathbb
C)\,\widetilde{}$\,-module}, Commun. Contemp. Math. {\bf 3}  (2001), 593-614.

\bibitem[M1]{M1} K.C.~Misra, {\em Structure of certain standard modules for
$A_n^{(1)}$ and the Rogers-Ramanujan identities}, J. Algebra {\bf 88} (1984),
196-227.

\bibitem[M2]{M2} K.C.~Misra, {\em Structure of some standard modules for
$C_n^{(1)}$}, J. Algebra {\bf 90} (1984), 385-409.

\bibitem[M3]{M3} K.C.~Misra, {\em Realization of the level two standard
$s\ell(2k+1, {\mathbb C})^\sim$-modules}, Trans. Amer. Math. Soc. {\bf 316}
(1989), 295-309.

\bibitem[M4]{M4} K.C.~Misra, {\em Realization of the level one standard
$\widetilde C_{2k+1}$-modules}, Trans. Amer. Math. Soc. {\bf 321}
(1990), 483-504.

\bibitem [P1]{P1} M. Primc,
\textit{$(k,r)$-admissible configurations and intertwining
operators}, Contemporary Math \textbf{442} (2007), 425--434.

\bibitem[P2]{P2} M. Primc, \textit{ Combinatorial bases of modules
for affine Lie algebra $B_2^{(1)}$}, Cent. Eur. J. Math. {\bf 11} (2013), 197-225. 

\bibitem[P\v S]{P\v S} M. Primc and T. \v Sikic, \textit{Combinatorial bases of basic modules for affine Lie algebras $C_n^{(1)}$}, J.Math.Phys., {\bf 57}, No.9 (2016), 1-19.

\bibitem [S]{S}
I. Siladi\'{c},  \textit{Twisted ${\mathfrak sl}(3,\mathbb
C)\,\widetilde{}$\,-modules and combinatorial identities}, 
Glasnik Matematicki Vol. 52 (72) (2017), 55-79.

\bibitem [T]{T} G. Trup\v cevi\' c,
\textit{Combinatorial bases of Feigin-Stoyanovsky's type subspaces
of higher-level standard $\tilde{\mathfrak sl}(\ell+1,\mathbb
C)$-modules},  J. Algebra {\bf  322} (2009), 3744--3774.

\bibitem[W]{W}
S. O. Warnaar, {\em The $A_{2n}^{(2)}$ Rogers-Ramanujan identities},  arXiv:1309.5216.

\bibitem[X]{X} C.~Xie, {\em Structure of the level two standard modules for the
affine Lie algebra $A_2^{(2)}$}, Comm. Algebra {\bf 18} (1990), 2397-2401.

\end{thebibliography}
\end{document}